\documentclass[11pt]{article}
\usepackage[margin=1.1in]{geometry}
\usepackage{amsmath,amssymb,amsthm}
\usepackage{booktabs,array}
\usepackage{graphicx}
\usepackage[hidelinks]{hyperref}
\hypersetup{pdftitle={Recurrent two-dimensional sequences over three and four letters simulating von Neumann constructions of finite ordinals},
pdfauthor={Bogdan Dumitru, Xiaofang Jiang, Mihai Prunescu}}
\newtheorem{thm}{Theorem}
\newtheorem{lem}[thm]{Lemma}
\newtheorem{prop}[thm]{Proposition}
\newtheorem{cor}[thm]{Corollary}
\theoremstyle{definition}
\newtheorem{defn}[thm]{Definition}
\theoremstyle{remark}
\newtheorem{rem}[thm]{Remark}

\newenvironment{alphenum}{\begin{enumerate}}{\end{enumerate}}
\newenvironment{romenum}{\begin{enumerate}}{\end{enumerate}}
\newenvironment{Denum}[1]{\begin{enumerate}\setcounter{enumi}{#1}}{\end{enumerate}}
\newenvironment{Fenum}[1]{\begin{enumerate}\setcounter{enumi}{#1}}{\end{enumerate}}
\newenvironment{Penum}{\begin{enumerate}}{\end{enumerate}}

\newcommand{\F}{\mathbb{F}_3}
\newcommand{\N}{\mathbb{N}}

\newcommand{\cl}[2]{[#1,#2]}
\newcommand{\Lam}{\Lambda}
\newcommand{\Gam}{\Gamma}
\newcommand{\B}{\mathsf{B}}
\newcommand{\cE}{\mathcal{E}}
\newcommand{\cB}{\mathfrak{B}}
\newcommand{\cx}{c_{x}}
\newcommand{\cy}{c_{y}}
\newcommand{\cz}{c_{z}}
\newcommand{\trip}[3]{(#1,#2,#3)}

\begin{document}
\title{Recurrent two-dimensional sequences over three and four letters simulating von Neumann constructions of finite ordinals} 
\author{Xiaofang Jiang\footnote{School of Mathematical Sciences and Hebei Center for Applied
Mathematics, Hebei Normal University, Shijiazhuang, China. E-mail: \texttt{jiangxf@hebtu.edu.cn}.}
\and Mihai Prunescu\footnote{Research Center for Logic, Optimization and
Security (LOS), Faculty of Mathematics and Computer Science, University of
Bucharest, Academiei 14, Bucharest (RO-010014), Romania.}\,\,\footnote{Simion
Stoilow Institute of Mathematics of the Romanian Academy, Research unit 5,
P. O. Box 1-764, Bucharest (RO-014700), Romania.
E-mail: \texttt{mihai.prunescu@imar.ro}.} \\
\and {Bogdan Dumitru}\footnote{University of Bucharest, Academiei 14,
Bucharest (RO-010014), Romania.\\
E-mail: \texttt{bogdan.dumitru@fmi.unibuc.ro}.}}
\date{}
\maketitle

\begin{abstract}
We study three recurrent double sequences whose coloured diagrams encode mirrored
versions of the finite-ordinal construction. Two sequences use three letters and
realise variants with two and three initial symbols. A four-letter sequence realises
the von Neumann construction from one initial symbol. For each sequence we describe
its cells explicitly and identify the ordered nesting of its rectangles. For four
letters, we distinguish ordinary and virtual boxes and prove their assembly rules.
The proofs proceed by induction along anti-diagonals. For the constructions with
one and three initial symbols, we reduce the recurrence to finite sets of local
configurations and verify those sets by computer. We derive
formulas for the numbers of nested boxes, including arithmetic-term representations, and we prove that none of the three sequences is $k$-automatic for any integer $k\ge2$.
\end{abstract}


\section{Introduction}

A recurrent double sequence over a finite set $A$ is a function $t:\N^2\to A$ determined by
its values on the two axes and by a rule $t(x,y)=f(t(x-1,y),t(x-1,y-1),t(x,y-1))$ for
$x,y\ge1$, where $f:A^3\to A$. Read along the anti-diagonals $x+y=s$, such a sequence is the
space-time diagram of a one-dimensional cellular automaton with a boundary
\cite{Kari05,Willson84}. This paper studies two sequences over $\F=\{0,1,2\}$ whose pictures
encode constructions from two and three initial symbols, and one sequence over
four letters whose picture encodes the construction from one initial symbol, which is John von
Neumann's construction of the finite ordinals.

\subsection{Previous work}
Recurrent double sequences over finite sets can simulate Turing machines, so their global
behaviour is undecidable in general: there is no algorithm that decides whether such a
sequence is ultimately zero \cite{Pru08}, an analogue of the undecidability of tiling the
plane \cite{Robinson71}. When the rule $f$ is linear over a finite field, the sequence is
automatic in the sense of Cobham and of Allouche and Shallit
\cite{Cobham72,AS03,AvHPS96,AvHPPS97}, and more generally $\mathbb{F}_p$-affine recurrent
sequences over $\mathbb{F}_q$ are $p$-automatic \cite{Pru13}. Linear rules produce
self-similar sequences generated by systems of substitutions, as shown for linear rules over
finite fields \cite{Pru09}, for homomorphisms of finite abelian
$p$-groups \cite{Pru11a}, for linear rules over the matrix ring $M_2(\mathbb{F}_2)$
\cite{Pru11b}, and for the Pascal rule modulo $2$ with Thue--Morse boundary values
\cite{Pru11c}. Some non-linear rules are self-similar as well, and \cite{Pru10} proves that a
recurrent double sequence which agrees on a large enough initial square with a sequence
produced by context-free substitutions agrees with it everywhere. Other sequences are not
automatic at all, and
\cite{Pru14} gives such a sequence with values in $\{0,1\}$, whose rows are periodic with
periods that grow with the row. Between these two ends lie sequences whose structure is
recursive but not self-similar in the sense of substitutions. According to its abstract, the
preprint \cite{Pru15} displays sequences over $\{0,1,2\}$ with constant boundary values that
model binary counting, the successor function and the finite ordinals, and observes that three
values suffice for these purposes.

\subsection{The constructions that are simulated}\label{sec:constructions}
From one initial symbol $\emptyset$, define the words
\[
s_0=\emptyset,\qquad s_n=(s_{n-1}\cdots s_0\mid s_0\cdots s_{n-1})\quad(n\ge1).
\]
For $p\in\{2,3\}$, start instead with $p$ distinct atoms $s_i=\emptyset_i$ for
$1\le i\le p$ and put
\[
s_n=(s_{n-p}\cdots s_1\mid s_1\cdots s_{n-p})\quad(n>p).
\]
These are the \emph{mirrored ordinals} with one, two and three initial symbols.
Ignoring order, repetition and the bar in the first construction gives the von Neumann
ordinals $n=\{0,\ldots,n-1\}$ \cite{vN23,Cantor97}.

The \emph{occurrence tree} of a word has the word as its root. Each occurrence in the
defining list gives a separate child, in the listed order, and each child is expanded
by the same rule. Atoms are leaves.
A sequence \emph{realises} a construction if the ordered tree of its bordered rectangles
inside each main rectangle is isomorphic to the occurrence tree of the corresponding word.
Our examples are $t$ (three letters, two initial symbols), $\rho$ (three letters, three
initial symbols), and $\tau$ (four letters, one initial symbol).

\subsection{Results}
Each of the sequences $t$, $\rho$ and $\tau$ uses a finite alphabet $A$, a rule
$f:A^3\to A$ and a constant boundary value $c\in A$. The corresponding sequence
$\alpha:\N^2\to A$ is defined by
\[
\begin{aligned}
\alpha(x,0)=\alpha(0,y)&=c &&(x,y\ge0),\\
\alpha(x,y)&=f(\alpha(x-1,y),\alpha(x-1,y-1),\alpha(x,y-1)) &&(x,y\ge1).
\end{aligned}
\]

Theorem~\ref{thm:main} describes $t(x,y)$ from its coordinates independently
of the cell recurrence. Section~\ref{sec:desc} states the description after the
necessary definitions. To prove the description, we check that it has the prescribed
boundary values and satisfies the recurrence at each cell. Induction along the
anti-diagonals then gives equality with $t$, because each cell depends only on the
two preceding anti-diagonals. Section~\ref{sec:proof} divides the recurrence check into
finitely many cases, each requiring a single evaluation of $f$.

In coordinates, Theorem~\ref{thm:structure} lists the vertices of all white-bordered
rectangles, main and nested, the position of the switching square inside each rectangle, and
the structure of the transition zones between consecutive rectangles. Its
Corollary~\ref{cor:ordinals} identifies the ordered box trees with the occurrence trees of
the words $(s_n)$.

The rule $f$ itself is the subject of Theorem~\ref{thm:triples}. Only $23$ of the $27$ triples
of $\F^3$ occur as arguments of $f$ in the computation of $t$, so $f$ can be changed on the
other $4$ triples without changing $t$. This gives $81$ functions, and $27$ symmetric ones,
that generate the same sequence, and no smaller set of triples determines $t$.

The finite claims that were also verified by computer are stated in
Section~\ref{sec:computer} together with the ranges verified. The proofs in
Sections~\ref{sec:proof}--\ref{sec:triples} do not depend on these computations.

Section~\ref{sec:three-initial} gives the construction with three initial symbols. Its
polynomial rule and its description differ from those of $t$, but the proof follows the
same induction along anti-diagonals. Lemma~\ref{lem:three-initial-local} reduces its
local recurrence check to a finite catalogue, which is verified by computer.

Section~\ref{sec:four} treats the construction with one initial symbol.
Theorem~\ref{thm:four} gives a sequence over four letters whose white-bordered rectangles
realise this construction, and describes all its cells. Theorem~\ref{thm:three-none}
excludes the prescribed arrangement over three letters within the stated finite windows,
for symmetric rules with constant boundary values and the stated density and nesting conditions.

The boxes of the four-letter sequence are built recursively. A box of level $n\ge2$
contains two ordinary boxes of level $n-1$ and, between them, a virtual box of the same
level whose walls are dashed. At level $1$, the middle region is the entirely white
virtual box of level $0$ (Proposition~\ref{prop:triple}). Unfolding the recursion gives
the children $S_{n-1},\dots,S_0,S_0,\dots,S_{n-1}$ of a box of level $n$.
Along the diagonal, the main boxes begin with two boxes of level $1$. For each $j\ge0$,
the next group consists of the children of a virtual box $V_j$ of level $3j+3$,
followed by ordinary boxes of levels $3j+3$ and $3j+4$
(Proposition~\ref{prop:diagonal}).

As for $t$, the proof of Theorem~\ref{thm:four} proceeds by induction along the
anti-diagonals. The recursive structure reduces the recurrence check to the local
configurations in a bounded region, where they were verified by computer.

In all three constructions the nesting gives short
recurrences for the number of boxes, with the closed expressions of
Corollaries~\ref{cor:two-count}, \ref{cor:three-count} and~\ref{cor:four-count}.
For three initial symbols, Corollary~\ref{cor:three-count} gives both a binomial sum and
an arithmetic term obtained by extracting a coefficient from the generating function.

Theorem~\ref{thm:nonauto} shows that none of $t$, $\rho$ and $\tau$ is
$k$-automatic for any $k\ge2$. For each sequence, we choose a letter on the line
$y=x+1$ whose gaps are unbounded but have length at most a constant times the
logarithm of their starting position, for sufficiently large starting positions.
An automatic sequence cannot have this gap behaviour.

\section{Three letters and two initial symbols}
Throughout, $t:\N^2\to\F$ is the recurrent double sequence with $t(x,0)=t(0,y)=1$ for all
$x,y\ge0$ and
\begin{align*}
t(x,y)&=f\big(t(x-1,y),\,t(x-1,y-1),\,t(x,y-1)\big),\\
f(x,y,z)&=x^2y^2+y^2z^2+y^2+xz+2y+2\quad\text{in }\F.
\end{align*}
We call the value $0$ white, $1$ red and $2$ blue. Figure~\ref{fig:three-overview} shows the
corner and the first five main rectangles, with a monochrome view of their borders beside it.
Figure~\ref{fig:S5} shows the fifth rectangle in more detail.

\begin{figure}[ht]
\centering
\includegraphics[width=0.48\textwidth]{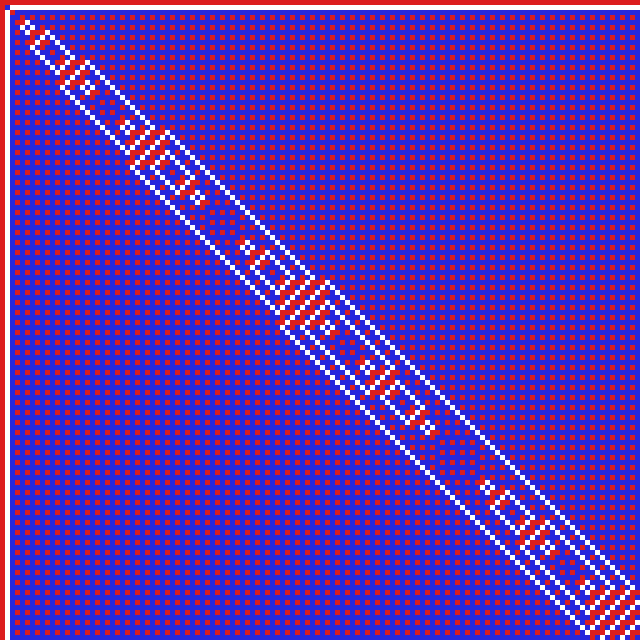}\hfill
\includegraphics[width=0.48\textwidth]{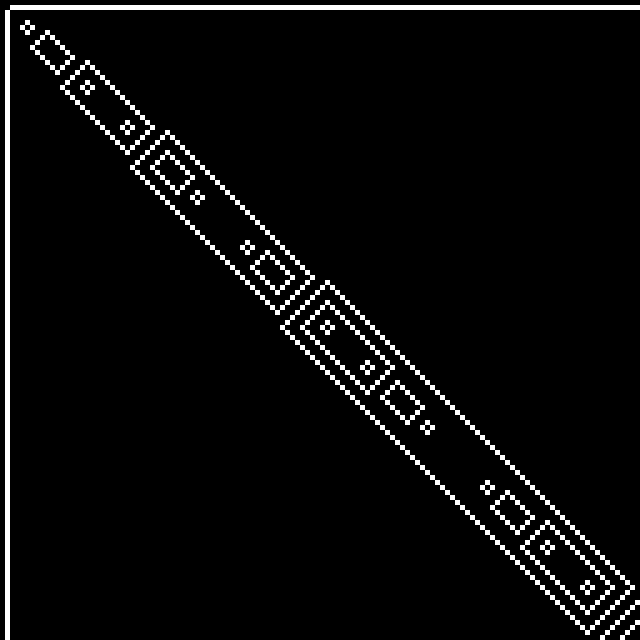}
\caption{The cells $(x,y)$ with $0\le x,y<128$ of the three-letter sequence $t$.
Left: white $0$, red $1$ and blue $2$. Right: the same cells, with red and blue drawn in black.
The main rectangles $S_1,\dots,S_5$ are shown completely. The opening of $S_6$ enters at the
bottom right. The monochrome view displays the white borders and their nesting.}
\label{fig:three-overview}
\end{figure}

\section{The rule and the rotated coordinates}\label{sec:rot}

For a cell $c=(x,y)$ with $x,y\ge1$ we write
\[
\cx=(x-1,y),\qquad \cy=(x-1,y-1),\qquad \cz=(x,y-1)
\]
for its three parents, so that $t(c)=f(t(\cx),t(\cy),t(\cz))$.

\begin{lem}\label{lem:rule}
For $x,y,z\in\F$,
\[
f(x,y,z)=\begin{cases}(x-z)^2-y & \text{if } y\ne0,\\ xz+2 & \text{if } y=0.\end{cases}
\]
Explicitly, $f(x,1,z)$ is $2$ if $x=z$ and $0$ otherwise, $f(x,2,z)$ is $1$ if $x=z$ and $2$
otherwise, and $f(x,0,z)$ is $2$ if $xz=0$, $0$ if $x=z\ne0$, and $1$ if $\{x,z\}=\{1,2\}$.
\end{lem}
\begin{proof}
If $y\ne0$ then $y^2=1$, so $f=x^2+z^2+xz+2y+3=(x-z)^2+3xz+2y+3=(x-z)^2-y$ in $\F$. If $y=0$
then $f=xz+2$.
\end{proof}

Since $f(x,y,z)=f(z,y,x)$, induction on $x+y$ shows that $t(x,y)=t(y,x)$ for all $x,y$.

\subsection{Rotated coordinates}
For a cell $(x,y)$ put $s=x+y$ and $d=y-x$, and write $\cl sd$ for the cell
$\big(\tfrac{s-d}2,\tfrac{s+d}2\big)$. A pair $(s,d)$ is a cell if and only if $|d|\le s$ and
$s\equiv d\pmod2$. We call the set of cells with a fixed $s$ the \emph{row} $s$, since the
induction runs over $s$, and we think of $s$ as time and of $d$ as position. On odd rows all
positions are odd, and on even rows all positions are even. For $c=\cl sd$ with $|d|\le s-2$,
\[
\cx=\cl{s-1}{d+1},\qquad \cy=\cl{s-2}{d},\qquad \cz=\cl{s-1}{d-1},
\]
and the cells with $|d|=s$ are the cells on the axes, with value $1$. We write
$T\cl sd=t\big(\tfrac{s-d}2,\tfrac{s+d}2\big)$. The main diagonal $x=y$ is the line $d=0$, and
the symmetry $(x,y)\mapsto(y,x)$ is $d\mapsto-d$.

\begin{lem}[uniqueness]\label{lem:unique}
Let $D$ be a function on cells with $D\cl s{\pm s}=1$ for all $s$ and
$D(c)=f(D(\cx),D(\cy),D(\cz))$ for every cell $c=\cl sd$ with $|d|\le s-2$. Then $D=T$.
\end{lem}
\begin{proof}
Both $D$ and $T$ satisfy the same recurrence with the same boundary values, so they agree on
row $s$ if they agree on rows $s-1$ and $s-2$. Induction on $s$.
\end{proof}

We say that $D$ is \emph{correct at} $c$ if $D(c)=f(D(\cx),D(\cy),D(\cz))$. If $D$ is
symmetric, that is $D\cl s{-d}=D\cl sd$, then $D$ is correct at $\cl sd$ if and only if it is
correct at $\cl s{-d}$, because $f$ is symmetric in $x$ and $z$. On the diagonal $d=0$ the
parents $\cx$ and $\cz$ are mirror images, so $D(\cx)=D(\cz)$ for a symmetric $D$.

\begin{figure}[ht]
\centering
\includegraphics[width=0.46\textwidth]{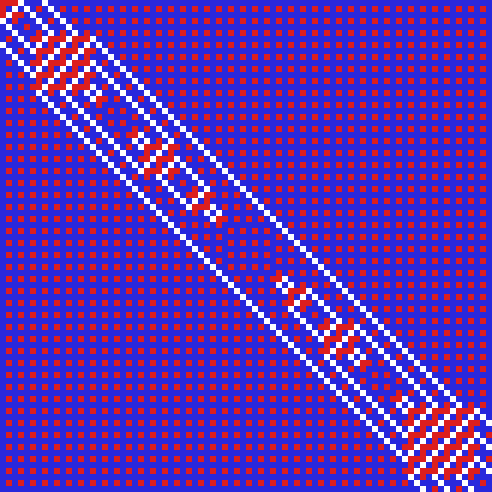}\hfill
\includegraphics[width=0.50\textwidth]{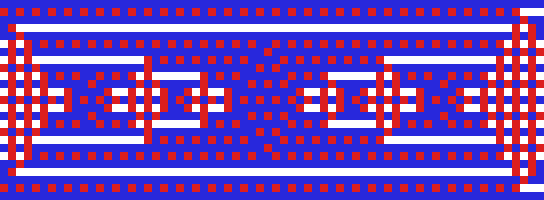}
\caption{Left: the cells $(x,y)$ with $50\le x,y<132$, where white is $0$, red is $1$ and blue
is $2$. The main rectangle $S_5$ occupies the middle, and the squares at its two ends are the
transition squares $T_4$ and $T_5$ of Proposition~\ref{prop:transition}. Right: the same
rectangle in the rotated coordinates of Section~\ref{sec:rot}, with the anti-diagonal index
$s=x+y$ grouped into successive pairs from left to right and the position $d=y-x$ increasing
from top to bottom. Each displayed column $p$ interleaves the two successive anti-diagonals
$s=2p$ and $s=2p+1$, using the first at even $d$ and the second at odd $d$. Reading down a
column therefore visits the two anti-diagonals alternately, tracing a staircase in the original
$(x,y)$ grid. The long white lines are the
walls of $S_5$, the short vertical white lines are opening and closing edges, and the diamond
of red cells in the middle is the switching square of Theorem~\ref{thm:structure}(iv).}
\label{fig:S5}
\end{figure} 

\section{The three lattices}\label{sec:lattices}

\begin{defn}\label{def:lattices}
For $\varepsilon\in\{0,2\}$ and a cell $\cl sd$ let
\[
\Lam_\varepsilon\cl sd=\begin{cases}1& \text{if } s \text{ is odd and } s\equiv|d|+\varepsilon \pmod 4,\\
2&\text{otherwise,}\end{cases}
\]
\[
\Gam\cl sd=\begin{cases}1& \text{if } s \text{ is even and } s\equiv d \pmod 4,\\ 2&\text{otherwise.}\end{cases}
\]
\end{defn}

In the coordinates $(x,y)$, $\Gam$ is red on the cells with $x$ and $y$ both even and blue
elsewhere. We call $\Gam$ the \emph{central lattice}. $\Lam_2$ is red on the cells with $x$ even and $y$ odd in the region $y<x$ and on the
cells with $x$ odd and $y$ even in the region $y>x$, because $s-|d|=2y$ for $y<x$ and $s-|d|=2x$
for $y>x$. $\Lam_0$ is red on the complementary cells of the odd rows. We call $\Lam_2$ the
\emph{normal environment} and $\Lam_0$ the \emph{reversed environment}. On even rows both are
blue, and on odd rows $\Lam_0$ is red where $\Lam_2$ is blue and conversely.

\begin{lem}\label{lem:lattices}
Let $c=\cl sd$ and let $\Pi$ be one of $\Gam,\Lam_0,\Lam_2$. Put $v=f(\Pi(\cx),\Pi(\cy),\Pi(\cz))$.
\begin{alphenum}
\item If $\Pi=\Gam$ then $v=\Gam(c)$.
\item If $\Pi=\Lam_\varepsilon$ and $d\ne0$ then $v=\Lam_\varepsilon(c)$.
\item If $\Pi=\Lam_\varepsilon$ and $d=0$ then $v=1$, while $\Lam_\varepsilon(c)=2$.
\end{alphenum}
\end{lem}
\begin{proof}
(a) Let $s$ be even. Then $\cx$ and $\cz$ lie on an odd row, so both are blue, and
$f(2,y,2)=1$ if and only if $y=2$ by Lemma~\ref{lem:rule}. Now $\Gam(\cy)=2$ if and only if
$s-2\not\equiv d\pmod 4$, which, as $s-d$ is even, is equivalent to $s\equiv d\pmod4$, that is
to $\Gam(c)=1$. Let $s$ be odd. Exactly one of $s-1-(d+1)$ and $s-1-(d-1)$ is divisible by
$4$, so one of $\cx,\cz$ is red and the other blue, $\cy$ is blue because $s-2$ is odd, and
$f(x,2,z)=2$ for $x\ne z$, which agrees with $\Gam(c)=2$.

(b) Let $s$ be odd. Then $\cx$ and $\cz$ are blue and $f(2,y,2)=1$ if and only if $y=2$, that
is if and only if $s-2\not\equiv|d|+\varepsilon\pmod4$, which is equivalent to
$s\equiv|d|+\varepsilon$, that is to $\Lam_\varepsilon(c)=1$. Let $s$ be even and $d\ne0$. Then
$|d|\ge2$, so $|d+1|$ and $|d-1|$ are $|d|+1$ and $|d|-1$ in some order, one of the two
odd-row parents is red and the other blue, $\cy$ is blue, and $f(x,2,z)=2$ for $x\ne z$.

(c) If $d=0$ then $\cx$ and $\cz$ have the same value and $\cy$ is blue, so $v=1$.
\end{proof}

Part (c) explains why the diagonal is never covered by an environment: at $d=0$ the two
lattices of the regions $y<x$ and $y>x$ meet and the recurrence produces a red cell instead of
the blue cell of the lattice.

\section{Boxes}\label{sec:boxes}

For $n\ge1$ put
\[
w_n=2n-1,\qquad L_n=2^{n+2}-6 .
\]

\begin{defn}\label{def:box}
For $n\ge1$ and odd $e\ge w_n$, the \emph{box of level $n$ opening at time $e$} is the set of cells
\[
\B(n,e)=\{\cl sd:\ e\le s\le e+L_n,\ |d|\le w_n\}.
\]
Its \emph{border} consists of the two \emph{walls} $|d|=w_n$, the \emph{opening edge} $s=e$ and
the \emph{closing edge} $s=e+L_n$. The other cells form its \emph{interior}. Its \emph{midpoint}
is $\mu=e+L_n/2=e+2^{n+1}-3$, its \emph{half-width} is $w_n$, and its \emph{vertex} is
$A=(u,v)=\big(\tfrac{e-w_n}2,\tfrac{e+w_n}2\big)$.
\end{defn}

In the coordinates $(x,y)$, $\B(n,e)$ is the closed rectangle with vertices $(u,v)$,
$(u+k,v+k)$, $(v+k,u+k)$ and $(v,u)$, where $k=L_n/2=2^{n+1}-3$. Its two edges parallel to the
diagonal have $k+1=2^{n+1}-2$ cells each and its two other edges have $v-u+1=2n$ cells each.

\begin{defn}[children]\label{def:children}
A box of level $1$ or $2$ has no children. For $n\ge3$ the \emph{children} of $\B(n,e)$ are the
$2(n-2)$ boxes
\[
\B(k,\,e+p_{n,k})\quad\text{and}\quad \B(k,\,e+q_{n,k}),\qquad 1\le k\le n-2,
\]
where
\[
p_{n,k}=2^{n+1}-2^{k+3}-2(n-k)+8,\qquad q_{n,k}=L_n-L_k-p_{n,k}=2^{n+1}+2^{k+2}+2(n-k)-8.
\]
In time order the children are the boxes of levels $n-2,n-3,\dots,1$, which form the
\emph{first half}, followed by the boxes of levels $1,2,\dots,n-2$, which form the \emph{second
half}.
\end{defn}

\begin{defn}\label{def:family}
The \emph{main boxes} are $S_n=\B(n,e_n)$ with $e_n=2^{n+2}-2n+3$ for $n\ge1$. The family
$\cB$ consists of the main boxes and of all boxes obtained from them by taking children
repeatedly.
\end{defn}

\begin{lem}\label{lem:timeline}
Let $\B(n,e)\in\cB$, $w=w_n$, $L=L_n$ and $\mu=e+L/2$.
\begin{alphenum}
\item $e\equiv w\pmod4$. Consequently $e$ is odd, $e+L\equiv w+2\pmod4$ and $\mu\equiv w+1\pmod4$.
\item Let $n\ge3$. The first child opens at $e+4$. Each child except the first copy of $S_1$ and
the last child closes $4$ rows before the next child opens. The first copy of $S_1$ closes at
$c_{\mathrm{gap}}=\mu-2n-1$, and the second copy opens at $e_{\mathrm{gap}}=\mu+2n+1$.
These times bound the switching gap. The last child closes at $e+L-4$.
The main box $S_n$ closes at $e_{n+1}-4$.
\item The children lie in the region $e+4\le s\le e+L-4$, $|d|\le w-4$. The first half lies in
$s\le\mu-2n-1$ and the second half in $s\ge\mu+2n+1$. Two boxes of $\cB$ are either nested or
disjoint, and two disjoint boxes are separated by at least $4$ rows.
\item Every cell of every box of $\cB$ satisfies $|d|\le s-8$.
\end{alphenum}
\end{lem}
\begin{proof}
(a) For a main box, $e_n=2^{n+2}-2n+3\equiv 3-2n\equiv 2n-1\pmod4$. For $n\ge3$ and
$1\le k\le n-2$ we have $p_{n,k}\equiv q_{n,k}\equiv 2(n-k)\pmod4$ and $L_n\equiv L_k\equiv2\pmod4$,
so a child of level $k$ of a box opening at $e\equiv 2n-1$ opens at a time
$\equiv 2n-1+2(n-k)\equiv 2k-1\pmod4$. The other congruences follow from $L\equiv 2$ and
$L/2=2^{n+1}-3\equiv1\pmod4$.

(b) $p_{n,n-2}=4$ and $p_{n,k-1}-p_{n,k}=2^{k+2}-2=L_k+4$. The second half is the mirror image
of the first half in time, because $q_{n,k}+L_k=L_n-p_{n,k}$. The first copy of $S_1$ closes at
$e+p_{n,1}+2=e+2^{n+1}-2n-4=\mu-2n-1$, and by the mirror symmetry the second copy opens at
$\mu+2n+1$. Finally $e_n+L_n+4=2^{n+3}-2n+1=e_{n+1}$.

(c) Children have half-width $w_k\le w_{n-2}=w-4$, and the time bounds follow from (b) by
induction over the nesting.

(d) A cell of a box inside $S_n$ satisfies $s-|d|\ge e_n-w_n=2^{n+2}-4n+4\ge8$.
\end{proof}

\begin{defn}[edges and hosts]\label{def:edges}
An \emph{edge} is a pair $(e,a)$ of a time and a half-width. For a box $H=\B(n,e)\in\cB$, the set
$\cE(H)$ consists of the edges $(e,w_n)$ and $(e+L_n,w_n)$ of $H$ and of the edges $(e',w_k)$
and $(e'+L_k,w_k)$ of each child $\B(k,e')$ of $H$. For the \emph{global host} $\ast$,
\[
\cE(\ast)=\{(3,1)\}\cup\{(e_n,w_n),\ (e_n+L_n,w_n):n\ge1\}.
\]
The pair $(3,1)$, which we call the \emph{virtual edge}, stands for the two white cells
$t(1,2)$ and $t(2,1)$. The \emph{host} of a cell $c$ is the smallest box of $\cB$ whose interior
contains $c$, and $\ast$ if there is no such box.
\end{defn}

By Lemma~\ref{lem:timeline}, two consecutive elements of $\cE(H)$ are $4$ rows apart,
with four exceptions.
\begin{itemize}
\item The opening and closing edges of a box of level $1$ are $2$ rows apart.
\item The opening and closing edges of a box of level $2$ are $10$ rows apart.
\item In a box of level $n\ge3$, the first copy of $S_1$ closes $4n+2$ rows before
the second copy opens. We call this interval the \emph{switching gap}.
\item The edges $(3,1)$ and $(9,1)$ of $\cE(\ast)$ are $6$ rows apart.
\end{itemize}
When two consecutive edges $(e,a)$ and $(e+4,b)$ are
$4$ rows apart, one of the five cases of Table~\ref{tab:J} holds.
\begin{table}[ht]
\centering
\begin{tabular}{lll}
\toprule
case & situation & $b$\\
\midrule
J1 & opening of $H$, then opening of its first child & $a-4$\\
J2 & closing of a first-half child, then opening of the next child & $a-2$\\
J3 & closing of a second-half child, then opening of the next child & $a+2$\\
J4 & closing of the last child, then closing of $H$ & $a+4$\\
J5 & in $\cE(\ast)$: closing of $S_n$, then opening of $S_{n+1}$ & $a+2$\\
\bottomrule
\end{tabular}
\caption{The five cases of two consecutive edges $(e,a)$ and $(e+4,b)$ of a set $\cE(H)$ that
are $4$ rows apart, and the half-width $b$ of the second edge in terms of the half-width $a$
of the first.}
\label{tab:J}
\end{table}

\section{The description of the sequence}\label{sec:desc}

\begin{defn}[the description $D$]\label{def:D}
For a cell $c=\cl sd$, the value $D(c)$ is given by the first of the following rules that
applies.
\begin{Denum}{-1}
\item\label{D0} If $|d|=s$, then $D(c)=1$. If $|d|=s-2\ge1$, then $D(c)=0$.
\item\label{D1} If $c$ lies on the border of a box of $\cB$, then $D(c)=0$.
\end{Denum}
Otherwise let $H$ be the host of $c$.
\begin{Denum}{1}
\item\label{D2} If $s=e\pm1$ and $|d|\le a-1$ for some $(e,a)\in\cE(H)$, then $D(c)=\Gam(c)$.
\item\label{D3} If $s=e\pm2$ and $|d|\le a-2$ for some $(e,a)\in\cE(H)$, then $D(c)=1$.
\item\label{D4} If $H=\B(n,e)$ with $n\ge2$ and $|s-\mu_H|+|d|\le w_n-3$, then
$D(c)=\Gam(c)$.
\item\label{D5} Otherwise $D(c)=\Lam_0(c)$ if $H$ is a box and $s<\mu_H$, and
$D(c)=\Lam_2(c)$ if $H=\ast$ or $s\ge\mu_H$.
\end{Denum}
\end{defn}

The cells of \ref{D2} and \ref{D3} are the \emph{transition cells}. Those of \ref{D2} lie on
even rows next to an edge, and those of \ref{D3} are the red cells on odd rows two steps from
an edge. The set of \ref{D4} is the \emph{switching diamond} of $H$, of radius $r_H=w_n-3$
around $\cl{\mu_H}0$. In the coordinates $(x,y)$ it is a square with sides parallel to the
axes. The cells of \ref{D5} are the \emph{environment cells}. Rules \ref{D0} and \ref{D1}
never apply to the same cell, by Lemma~\ref{lem:timeline}(d). Rule \ref{D2} concerns only even
rows, since $e$ is odd, and rule \ref{D3} only odd rows. The diamond is at distance at least
$5$ from every edge of $\cE(H)$. The function $D$ is symmetric under $d\mapsto-d$.

\begin{thm}\label{thm:main}
For all $x,y\ge0$, $t(x,y)=D\cl{x+y}{y-x}$.
\end{thm}

The picture has red axes, white lines $x=1$ and $y=1$, and the normal environment
$\Lam_2$ away from the diagonal. The main boxes $S_1,S_2,\dots$ lie along the diagonal.
For $n\ge3$, the interior of $S_n$ contains copies of $S_{n-2},\dots,S_1$ in the
reversed environment $\Lam_0$, followed by a switching square and copies of
$S_1,\dots,S_{n-2}$ in the normal environment. The first two boxes have no children.
Each transition cell lies one or two rows before or after an edge.

The phase of the environment inside a box does not depend on the depth of nesting: the
reversed environment fills the first half of each box and the normal environment the second
half. Remark~\ref{rem:copies} explains why this makes all copies of $S_k$ inside larger boxes
translates of the main box $S_k$.

\section{Proof of Theorem \ref{thm:main}}\label{sec:proof}

By Lemma~\ref{lem:unique} and rule \ref{D0} it suffices to prove that $D$ is correct at every
cell $c=\cl sd$ with $|d|\le s-2$. Each case below is one evaluation of $f$ by
Lemma~\ref{lem:rule}. We write the triple $(D(\cx),D(\cy),D(\cz))$ followed by its image
under $f$. By the symmetry of $D$ we may assume $d\ge0$ whenever convenient.

\subsection{Preliminary facts}
Throughout, $\B=\B(n,e)\in\cB$, $w=w_n$, $L=L_n$, and $\B'$ is the parent of $\B$ in $\cB$, or
$\ast$ if $\B$ is a main box. All congruences are modulo $4$, and Lemma~\ref{lem:timeline}(a)
is used without further mention.

\begin{Penum}
\item\label{P1} $D(c)=0$ if and only if $c$ is on a border or on a line $|d|=s-2\ge1$. Even rows
contain no border cells.
\item\label{P2} \emph{The columns beside a wall are blue.} Inside $\B$, $D\cl s{\pm(w-1)}=2$ for
even $s\in[e+1,e+L-1]$. Outside $\B$, $D\cl s{\pm(w+1)}=2$ for even $s\in[e-1,e+L+1]$.

Inside: the host is $\B$, because children have half-width at most $w-4$. Among the edges of
$\cE(\B)$ only the two edges of $\B$ itself reach $|d|=w-1$ in rule \ref{D2}, at the rows
$e+1\equiv w+1$ and $e+L-1\equiv w+1$, where $\Gam\cl s{w-1}=2$. The diamond has radius $w-3$,
and rule \ref{D5} gives $2$ on even rows. Outside: the host is $\B'$, and by
Lemma~\ref{lem:timeline}(c) the only elements of $\cE(\B')$ in the time interval $[e-2,e+L+2]$
are the two edges of $\B$, whose half-width $w$ is smaller than $w+2$. The diamond of $\B'$ lies
in the switching gap of $\B'$, and rule \ref{D5} gives $2$.
\item\label{P3} \emph{The transition rows.} Around the opening edge, $D\cl{e\pm1}d=\Gam\cl{e\pm1}d$
for $|d|\le w-1$ and $D\cl{e\pm2}d=1$ for $|d|\le w-2$. The same holds around the closing edge
with $e$ replaced by $e+L$. For the virtual edge, $D\cl20=2$ and $D\cl40=1$.

These cells are on no border, since the nearest other edges and walls are at least $4$ rows
or columns away. The host of each transition cell is $\B'$ before the opening and
after the closing, and $\B$ between the edges, because children open $4$ rows after
$\B$. The relevant edge belongs to $\cE$ of that host, so rule \ref{D2} applies on
even rows and rule \ref{D3} on odd rows.
\item\label{P4} $D\cl{e-2}{\pm w}=1$, that is, the two cells two rows before the corners of an
opening edge are red. The host is $\B'$ and the row is odd. If $\B$ is in the first half of $\B'$,
the element of $\cE(\B')$ preceding $(e,w)$ is $(e-4,w+4)$ in case J1 or $(e-4,w+2)$ in case J2,
and rule \ref{D3} applies. Otherwise rule \ref{D3} gives $1$ if it applies, and rule
\ref{D5} gives $\Lam_2\cl{e-2}w=1$ because $e-2\equiv w+2$.
\item\label{P5} $D\cl{e+L+2}{\pm w}=1$, that is, the two cells two rows after the corners of a
closing edge are red. If $\B$ is in the first half of $\B'$, so that the closing is followed by
case J2 or by the switching gap, rule \ref{D5} gives $\Lam_0\cl{e+L+2}{w}=1$ because
$e+L+2\equiv w$. Otherwise the next edge is $(e+L+4,b)$ with $b\in\{w+2,w+4\}$, by cases J3,
J4 and J5, and rule \ref{D3} applies.
\end{Penum}

\subsection{Border cells}
\begin{prop}\label{prop:borders}
$D$ is correct at every border cell.
\end{prop}
\begin{proof}
\emph{Wall cells} $c=\cl sw$ with $e+2\le s\le e+L$, which include the closing corner. By
\ref{P1} the parent $\cy$ is white, and by \ref{P2} the parents $\cx$ and $\cz$ are blue, so
the triple is $(2,0,2)$ and $f(2,0,2)=0$.

\emph{Inner cells of the opening edge}, $c=\cl ed$ with $|d|\le w-2$. The parent $\cy=\cl{e-2}d$
is red and $\cx,\cz$ are the $\Gam$-cells of row $e-1$ by \ref{P3}, of which one is red and the
other blue. The triple is $(1,1,2)$ or $(2,1,1)$, and $f=0$.

\emph{Opening corner} $c=\cl ew$. The parent $\cy=\cl{e-2}{w}$ is red by \ref{P4}, the parent
$\cz=\cl{e-1}{w-1}$ has $\Gam=1$ because $e-1\equiv w-1$, and the parent $\cx=\cl{e-1}{w+1}$ is
blue by \ref{P2}. The triple is $(2,1,1)$ and $f=0$.

\emph{Inner cells of the closing edge}, $c=\cl{e+L}d$ with $|d|\le w-2$. As for the opening
edge, using the rows $e+L-1$ and $e+L-2$ of \ref{P3}.
\end{proof}

\subsection{Transition cells}
\begin{prop}\label{prop:trans}
$D$ is correct at every transition cell.
\end{prop}
\begin{proof}
Let $(e,a)\in\cE(H)$ be the edge responsible for $c$ in rule \ref{D2} or \ref{D3}.

\emph{Rule \ref{D2} after an edge}: $c=\cl{e+1}d$ with $|d|\le a-1$. The parents $\cx$ and
$\cz$ lie on the edge and are white, and $\cy=\cl{e-1}d$ has the value $\Gam(\cy)$ by \ref{P3}.
By Lemma~\ref{lem:rule}, $f(0,1,0)=2$ and $f(0,2,0)=1$. Since $\Gam(\cy)=2$ if and only if
$e+1\equiv d$, which is the condition $\Gam(c)=1$, we get $D(c)=\Gam(c)$. The cell $\cl40$ is
of this kind, with the virtual edge.

\emph{Rule \ref{D2} before an edge}: $c=\cl{e-1}d$ with $|d|\le a-1$, and $c$ is not of the
previous kind. Then $(e,a)$ is the opening of a child of $H$, or the closing of $H$, in which
case $H$ has level at least $2$. The cell $\cl20$, before the virtual edge, is treated
directly: its parents are the axis cells $\cl1{\pm1}$ and $\cl00$, so the triple is $(1,1,1)$
and $f=2=\Gam\cl20$. In the remaining cases we distinguish the position of $c$.
\begin{romenum}
\item $|d|\le a-3$. Then $\cx$ and $\cz$ are red by \ref{P3} applied to row $e-2$, and
$f(1,y,1)$ is $2,1,0$ for $y=1,2,0$. It therefore suffices to show $D(\cy)=\Gam(\cy)$ for
$\cy=\cl{e-3}{d}$. In cases J1, J2, J3 and J5 the cell $\cy$ is a \ref{D2}-cell after the
previous edge $(e-4,a')$, and $a'\ge a-2$. In case J4 the same holds for $|d|\le a-5$, and for
$|d|=a-3$ \ref{P2} gives $D(\cy)=2=\Gam(\cy)$, since $e-3\equiv a-1$. Before the closing of a
box of level $2$, $\cy=\cl{\mu+2}0$ has $D(\cy)=2$ by class (E1) below and $\Gam(\cy)=2$.
\item $|d|=a-1$ with $d\ge0$, and $(e,a)$ is an opening. The parent $\cx=\cl{e-2}a$ is red by
\ref{P4}. The parent $\cz=\cl{e-2}{a-2}$ is red, by \ref{P3} if $a\ge3$ and as the mirror
image of $\cx$ if $a=1$. The parent $\cy=\cl{e-3}{a-1}$ is blue. In cases J1 and J2,
this follows from \ref{P3} because $\Gam=2$ as $e-3\equiv a+1$.
In cases J3 and J5, use \ref{P2}. After the switching gap,
$\cy=\cl{\mu+w_H-1}0$, and for $S_1$ at $e=9$, $\cy=\cl60$.
Both cells are blue by class (E1) below.
The triple is $(1,2,1)$, so $f=1=\Gam(c)$, since
$e-1\equiv a-1$.
\item $|d|=a-1$ with $d\ge0$, and $(e,a)$ is the closing of $H$, so $a\ge3$. The parent
$\cx=\cl{e-2}a$ is a wall cell, $\cz$ is red by \ref{P3}, and $\cy$ is blue by \ref{P2}. The
triple is $(0,2,1)$, so $f=2=\Gam(c)$, since $e-1\equiv a+1$.
\end{romenum}

\emph{Rule \ref{D3} after an edge}: $c=\cl{e+2}d$ with $|d|\le a-2$. The parent $\cy$ is on the
edge, and $\cx,\cz$ are $\Gam$-cells of row $e+1$, one red and one blue. The triple is
$(1,0,2)$ or $(2,0,1)$, and $f=1$.

\emph{Rule \ref{D3} before an edge}, not of the previous kind: $c=\cl{e-2}d$ with $|d|\le a-2$.
This happens in three situations. In cases J3 and J5 with $|d|=a-2$, and in case J4 with
$|d|=a-4$, the cell belongs to class (E3) below and is treated there. In case J4 with $|d|=a-2$,
the parent $\cy=\cl{e-4}{a-2}$ has $\Lam_2=2$, since $e-4\equiv a+2$, and $\cx,\cz$ are blue by
\ref{P2}, so the triple is $(2,2,2)$ and $f=1$. Before the closing of a box $\B(2,e')$ of level
$2$, $c=\cl{e'+8}{\pm1}$, the parent $\cy=\cl{e'+6}1$ has $\Lam_2=2$, the parent $\cz=\cl{e'+7}0$
is blue by class (E1), and $\cx$ is blue by \ref{P2}, so the triple is $(2,2,2)$ and $f=1$.
\end{proof}

\subsection{Switching diamonds}
\begin{prop}\label{prop:diamond}
$D$ is correct at every cell of a switching diamond.
\end{prop}
\begin{proof}
Let $H=\B(n,e)$ with $n\ge2$, $\mu=\mu_H$, $r=w_n-3$, and let $\delta(c)=|s-\mu|+|d|$. Then
$\delta$ takes even values only and $\mu-r\equiv0$.

If the three parents of $c$ lie in the diamond, Lemma~\ref{lem:lattices}(a) applies. For
$s>\mu$ the parents satisfy $\delta\le\delta(c)$, and for $s\le\mu$ and $\delta(c)\le r-2$ they
satisfy $\delta\le\delta(c)+2\le r$. It remains to treat the cells with $s\le\mu$ and
$\delta(c)=r$, for which $s\equiv|d|\pmod4$.

The parents of such a cell that lie outside the diamond lie on rows $\ge\mu-r-2$, later than
the transition rows of the edge that precedes the diamond, so they are environment cells with
value $\Lam_0$. Take $d\ge0$. If $s$ is even, then $\Gam(c)=1$, the parent $\cx=\cl{s-1}{d+1}$
has $\Lam_0=2$ because $s\not\equiv d+2$, the parent $\cy$ is blue because its row is even,
and $\cz$ is a diamond cell on an odd row if $d\ge2$ and the mirror image of $\cx$ if $d=0$, so
blue in both cases. The triple is $(2,2,2)$ and $f=1$. If $s$ is odd, then $\Gam(c)=2$, the
parent $\cx$ is blue because its row is even, the parent $\cz=\cl{s-1}{d-1}$ is in the diamond
with $\Gam=1$, and $\cy=\cl{s-2}{d}$ has $\Lam_0=2$ because $s-2\not\equiv d$. The triple is
$(2,2,1)$ and $f=2$.
\end{proof}

\subsection{Environment cells}
Let $c=\cl sd$ be an environment cell with host $H$. Its \emph{phase} $\varepsilon$ is $0$ if $H$
is a box and $s\le\mu_H$, and $2$ otherwise. On the even row $s=\mu_H$ both lattices give
$2$, so the convention does not change any value. A parent $p$ of $c$ is \emph{irregular} if
$D(p)\ne\Lam_\varepsilon(p)$. If $d\ne0$ and no parent is irregular, $D$ is correct at $c$ by
Lemma~\ref{lem:lattices}(b). The \emph{successors} of a cell $\cl{s'}{d'}$ are the cells
$\cl{s'+1}{d'\pm1}$ and $\cl{s'+2}{d'}$, which have it as a parent.

\begin{lem}\label{lem:hosts}
Let $c$ be an environment cell and $p$ a parent of $c$ that is not white. Then $p$ has the
same host as $c$.
\end{lem}
\begin{proof}
Suppose first that $p$ lies in the interior of a box $X$ whose interior does not contain $c$.
If $p=\cy$, then $e_X<s-2<e_X+L_X$ and $|d|<w_X$. Since $c$ is not white, $c$ is not on the
border of $X$, so $s>e_X+L_X$, hence $s=e_X+L_X+1$ and $c$ is a \ref{D2}-cell after the closing
of $X$, not an environment cell. If $p=\cx$ or $p=\cz$, then $|d|\le|d\pm1|+1\le w_X$ and
$e_X<s\le e_X+L_X$, so $c$ lies in $X$, and since $c$ is not white it lies in the interior of
$X$, a contradiction.

Suppose next that $c$ lies in the interior of a box $X$ and $p$ does not. Then $p$ is on the
border of $X$ and white, or $p=\cy$ with $s-2<e_X$, in which case $s=e_X+1$ and $c$ is a
\ref{D2}-cell after the opening of $X$.
\end{proof}

\begin{lem}\label{lem:Eclasses}
Let $c$ be an environment cell that has an irregular parent or lies on the diagonal $d=0$.
Then $c$ belongs to one of the classes (E1) to (E6) listed in the proof of
Proposition~\ref{prop:env}.
\end{lem}
\begin{proof}
By Lemma~\ref{lem:hosts}, an irregular parent $p$ of $c$ is white, or lies in the host $H$ of
$c$ and is one of the following: a cell of an axis, a transition cell, a diamond cell, or an
environment cell of the other phase. We go through the successors of each kind of cell and
keep those successors that are environment cells.

\emph{Axis cells.} The successors of $\cl{s'}{\pm s'}$ lie on the axes or on the white lines.

\emph{White lines.} The successors of $\cl{s'}{s'-2}$ are $\cl{s'+1}{s'-1}$ on the white line,
and $\cl{s'+1}{s'-3}$ and $\cl{s'+2}{s'-2}$, which have $|d|=s-4$ and are $\cl40$, a
\ref{D2}-cell, or cells of class (E6).

\emph{Walls.} Let $p=\cl{s'}{w}$ be a wall cell of $\B=\B(n,e)$ with $s'\in[e,e+L]$. The
successor $\cl{s'+1}{w+1}$ lies outside $\B$ on a row of $[e+1,e+L+1]$ and belongs to class
(E2). The successor $\cl{s'+1}{w-1}$ is a \ref{D2}-cell of $\B$ if $s'+1$ is $e+1$ or $e+L-1$, a
\ref{D2}-cell after the closing if $s'+1=e+L+1$, and otherwise belongs to class (E2). The
successor $\cl{s'+2}{w}$ is a wall cell unless $s'=e+L$, in which case it belongs to class (E3).

\emph{Inner edge cells.} The successors of an inner cell $\cl{e'}{d'}$, $|d'|\le a-2$, of an
edge $(e',a)$ are \ref{D2}-cells and \ref{D3}-cells after that edge.

\emph{Transition cells} of an edge $(e,a)\in\cE(H)$. The successors of the cells before the
edge, on rows $e-1$ and $e-2$, lie on the rows $e-1$ with $|d|\le a-1$, $e$ with $|d|\le a$,
and $e+1$ with $|d|\le a-1$, so they are \ref{D2}-cells, edge cells, or \ref{D2}-cells after
the edge. The successors of the cells after the edge, on rows $e+1$ and $e+2$, lie on the
rows $e+2$ with $|d|\le a$, $e+3$ with $|d|\le a-1$ and $e+4$ with $|d|\le a-2$. On row $e+2$
the cells with $|d|\le a-2$ are \ref{D3}-cells and the cells $\cl{e+2}{\pm a}$ are wall cells if
$(e,a)$ is an opening and cells of class (E3) if it is a closing. For the rows $e+3$ and $e+4$
let $(e+g,b)$ be the next element of $\cE(H)$.
\begin{itemize}
\item In cases J3, J4 and J5, where $g=4$ and $b\ge a+2$, the rows $e+3$ with $|d|\le b-1$ and
$e+4$ with $|d|\le b$ consist of \ref{D2}-cells and edge cells.
\item In case J2, where $b=a-2$, the only remaining cells on row $e+3$ are $|d|=a-1$, of class
(E4), and on row $e+4$ all cells with $|d|\le a-2=b$ are edge cells.
\item In case J1, where $b=a-4$, the remaining cells are $\cl{e+3}{\pm(a-3)}$ of class (E4),
$\cl{e+3}{\pm(a-1)}$ of class (E2), and $\cl{e+4}{\pm(a-2)}$ of class (E4).
\item After the opening of a box of level $2$, where $a=3$, row $e+3$ gives $\cl{e+3}0$ of
class (E1) and $\cl{e+3}{\pm2}$ of class (E2). Row $e+4$ gives $\cl{e+4}{\pm1}$, whose
non-white parents on rows $e+2$ and $e+3$ are regular, because $D\cl{e+2}{\pm1}=1=\Lam_0$ as
$e+2\equiv1$, and $D\cl{e+3}{0}=D\cl{e+3}{\pm2}=2$.
\item After the opening of a box of level $1$, where $a=1$, the rows $e+1$ and $e+2$ belong to
the box, and the next cell $\cl{e+3}0$ is a \ref{D2}-cell after its closing.
\item After the closing of the first copy of $S_1$, before the switching gap, row $e+3$ gives
$\cl{e+3}0$ of class (E1). There are no \ref{D3}-cells after an edge of half-width $1$.
The cells $\cl{e+2}{\pm1}$ are regular, with value $1=\Lam_0$ by \ref{P5}, so their
successors on row $e+4$ give no further case.
\item For the virtual edge $(3,1)$, the successors of $\cl40$ are $\cl5{\pm1}$ of class (E6)
and $\cl60$ of class (E1).
\end{itemize}

\emph{Diamond cells.} Let $p$ be in the diamond of $H$, so $\delta(p)\le r=w_n-3$ in the
notation of Proposition~\ref{prop:diamond}, and let $c$ be a successor of $p$ outside the
diamond. If $c$ lies on a row $\le\mu$, then $\delta(c)\le\delta(p)$, which is impossible, so
$s>\mu$. Then $\delta(c)\le\delta(p)+2\le r+2$, and since $\delta$ is even and $\delta(c)>r$, we
get $\delta(c)=r+2=w_n-1$, which is class (E5).

\emph{Other phase.} With our convention, a parent of the other phase occurs only for $c$ on
row $\mu+1$ with $\cy$ on row $\mu-1$. Outside the diamond, row $\mu+1$ contains only the
cells $|d|=w_n-2$, which belong to class (E5).

\emph{The diagonal.} Let $c=\cl s0$ be an environment cell with host $H=\B(n,e)$, so that $s$ is
not in the time span of a child. In a gap of $4$ rows between consecutive elements
$(e',\cdot)$ and $(e'+4,\cdot)$ of $\cE(H)$, the rows $e'+1$ and $e'+3$ are \ref{D2}-rows that
reach $d=0$. In the switching gap $(c_{\mathrm{gap}},e_{\mathrm{gap}})$, or in the gap $(e,e+10)$ for $n=2$, the even
rows are the \ref{D2}-rows $c_{\mathrm{gap}}+1$ and $e_{\mathrm{gap}}-1$, the diamond rows from $\mu-r$ to $\mu+r$, and
the two rows $c_{\mathrm{gap}}+3=\mu-w_n+1$ and $e_{\mathrm{gap}}-3=\mu+w_n-1$, whose diagonal cells belong to class
(E1). In a box of level $1$ the only such row is a \ref{D2}-row. For $H=\ast$ the gaps between
consecutive edges are gaps of $4$ rows, except the gap $(3,9)$, whose even rows $2,4,8$ are
\ref{D2}-rows and whose row $6$ gives $\cl60$ of class (E1).
\end{proof}

\begin{prop}\label{prop:env}
$D$ is correct at every environment cell.
\end{prop}
\begin{proof}
By Lemma~\ref{lem:Eclasses} it suffices to treat the classes (E1) to (E6). All other
environment cells are covered by Lemma~\ref{lem:lattices}(b).

\textbf{(E1)} \emph{Environment cells on the diagonal.} These are $\cl{\mu-w_n+1}0$ and
$\cl{\mu+w_n-1}0$ in each box of level $n\ge2$, and the cell $\cl60$. For $n\ge3$,
$\cl{\mu-w_n+1}0=\cl{c_{\mathrm{gap}}+3}0$, where $c_{\mathrm{gap}}$ is the closing time of the first copy of $S_1$. Its
parents are $\cl{c_{\mathrm{gap}}+2}{\pm1}$, red by \ref{P5}, and $\cl{c_{\mathrm{gap}}+1}0$, which has $\Gam=1$ since
$c_{\mathrm{gap}}\equiv3$. The triple is $(1,1,1)$ and $f=2=\Lam_0$. For $n=2$ the same triple arises,
using \ref{P3}. The cell $\cl{\mu+w_n-1}0$ belongs to class (E5). For $\cl60$ the parents
$\cl5{\pm1}$ have $\Lam_2=2$ and $\cl40$ is red, so the triple is $(2,1,2)$ and $f=2$.

\textbf{(E2)} \emph{Next to walls.} Inside: $c=\cl s{w-1}$ with $s$ even and $e+3\le s\le e+L-3$.
The parent $\cx$ is a wall cell, $\cy$ is blue by \ref{P2}, and $\cz$ is not white, so
$f(0,2,z)=2$ for $z\ne0$, which is $D(c)$. Outside: $c=\cl s{w+1}$ with $s$ even and
$e+1\le s\le e+L+1$. The parent $\cz$ is a wall cell or a corner, $\cy$ is blue by \ref{P2}, and
$\cx$ is not white, so $f(x,2,0)=2$ for $x\ne0$, which is $D(c)$ by \ref{P2}.

\textbf{(E3)} \emph{After a closing corner}: $c=\cl{e+L+2}{w}$, with $D(c)=1$ by \ref{P5}. The
parent $\cy$ is the corner and white, $\cx$ is blue by \ref{P2}, and $\cz=\cl{e+L+1}{w-1}$ has
$\Gam=1$ since $e+L+1\equiv w-1$. The triple is $(2,0,1)$ and $f=1$.

\textbf{(E4)} \emph{Behind transition rows.} In case J1, with the opening of $H$ at $e$ and
$a=w$, the cell $c=\cl{e+3}{w-3}$ has the parents $\cl{e+2}{w-2}$ and $\cl{e+2}{w-4}$, both red,
and $\cl{e+1}{w-3}$ with $\Gam=1$. The triple is $(1,1,1)$ and $f=2$. Next, $c=\cl{e+4}{w-2}$
has the parents $\cl{e+3}{w-1}$ and $\cl{e+3}{w-3}$, both blue, and $\cl{e+2}{w-2}$, red. The
triple is $(2,1,2)$ and $f=2=\Lam_0$, since $e+4\equiv w$. The cell $\cl{e+3}{w-1}$ belongs to
class (E2). In case J2, with the closing of a child $X$ at $c_0$ and half-width $a$, the cell
$c=\cl{c_0+3}{a-1}$ has the parents $\cl{c_0+2}a$, red by \ref{P5}, $\cl{c_0+2}{a-2}$, red, and
$\cl{c_0+1}{a-1}$ with $\Gam=1$. The triple is $(1,1,1)$ and $f=2$. In cases J3, J4 and J5
the rows behind the transition rows are transition rows or edges of the next edge, so there is
nothing to check. After the opening of a box of level $2$ and in the switching gap, the cells
in question belong to classes (E1) and (E2).

\textbf{(E5)} \emph{Below a diamond}: $c$ with $s>\mu$ and $(s-\mu)+|d|=w_n-1$, with $d\ge0$.
If $s$ is even, then $s\equiv d+2$, the parent $\cy$ is in the diamond with $\Gam=1$, the
parent $\cx$ has $\Lam_2=2$, and $\cz$ is a diamond cell on an odd row, or the mirror image of
$\cx$ when $d=0$, so blue in both cases. The triple is $(2,1,2)$ and $f=2=\Lam_2(c)$. If $s$ is
odd, then $s\equiv d$, the parent $\cz$ is in the diamond with $\Gam=1$, the parent $\cx$ is
blue, and $\cy$ is a diamond cell on an odd row if $s\ge\mu+3$, while for $s=\mu+1$ it is the
environment cell $\cl{\mu-1}{w_n-2}$ with $\Lam_0=2$. The triple is $(2,2,1)$ and
$f=2=\Lam_2(c)$.

\textbf{(E6)} \emph{Next to the white lines}: $c=\cl s{\pm(s-4)}$ with $s\ge5$, that is, $x=2$ or
$y=2$. The parents $\cx$ and $\cy$ are on the white line, so $f(0,0,z)=2=\Lam_2(c)$.
\end{proof}

\subsection{The white lines}
For $c=\cl s{s-2}$ with $s\ge4$, the parents $\cx$ and $\cy$ are on the red axis and $\cz$ is on
the white line, so the triple is $(1,1,0)$ and $f=0$. For $c=\cl31$, the parent $\cz=\cl20$ is
blue, so the triple is $(1,1,2)$ and $f=0$.

Propositions~\ref{prop:borders} to \ref{prop:env} and the white lines cover every cell with
$|d|\le s-2$, so $D$ is correct everywhere and Theorem~\ref{thm:main} follows from
Lemma~\ref{lem:unique}. \qed

\begin{rem}[why inner boxes are exact copies]\label{rem:copies}
The recurrence determines the interior of a box $\B(n,e)$ from its opening edge and
the neighbouring cells. By \ref{P3}, the preceding row $\cl{e-1}{d}$, $|d|\le w-1$,
is coloured by the central lattice $\Gam$. By \ref{P2}, the columns $|d|=w\pm1$ are
blue. These conditions do not depend on the position of the box, so every box of
level $k$ in $\cB$ is a translate of $S_k$. In particular, the environment inside a
box has the same phase at every depth of nesting.
\end{rem}

\section{The structure in coordinates}\label{sec:coords}

An \emph{ordinal} of a double sequence is a set of cells that forms a closed rectangle with
vertices $(u,v)$, $(u+k,v+k)$, $(v+k,u+k)$ and $(v,u)$, where $u<v$ and $k\ge1$, all of whose
border cells are white.

\begin{cor}[ordinals are boxes]\label{cor:boxes}
The ordinals of $t$ are the boxes of $\cB$. The ordinals contained in no other are the
main boxes $S_n$, $n\ge1$, and the ordinals directly inside a box, with no ordinal in between,
are its children.
\end{cor}
\begin{proof}
Let $R$ be an ordinal with half-width $W=v-u$ and time span from $s_0$ to $s_0+2k$. By
\ref{P1} the cells of its walls are border cells of boxes or cells of the white lines. A cell of
a white line has no white cell two rows before or after it in the same column, so for $k\ge1$
the wall cells are box cells. A wall of $R$ cannot contain an inner cell of an edge of a box
wider than $R$, because the cells two rows before and after such a cell are red by \ref{P3}. So
the walls of $R$ consist of wall cells of boxes of half-width $W$. Two such boxes are at least
$4$ rows apart, and the cells between them in the columns $|d|=W$ are not white by \ref{P4}
and \ref{P5}, so all wall cells of $R$ belong to one box $\B$. Inside $\B$ no row except the two
edges is white on all of $|d|\le W$, because children have half-width at most $W-4$, and for
$W=1$ the only odd rows of $\B$ are its edges. Hence $R=\B$.
\end{proof}

\begin{thm}[structure]\label{thm:structure}
The sequence $t$ is symmetric and consists of the following parts.
\begin{romenum}
\item The axes $x=0$ and $y=0$ are red. The lines $x=1$ and $y=1$ are white except for
$t(1,0)=t(0,1)=1$ and $t(1,1)=2$.
\item The main ordinal $S_n$, $n\ge1$, has vertices $A_n$, $B_n$, $\sigma(B_n)$, $\sigma(A_n)$
with $\sigma(u,v)=(v,u)$ and
\[
A_n=(2^{n+1}-2n+2,\ 2^{n+1}+1),\qquad B_n=(2^{n+2}-2n-1,\ 2^{n+2}-2).
\]
Its edges parallel to the diagonal have $2^{n+1}-2$ cells and its other edges have $2n$ cells.
The vertices of every ordinal have coordinates congruent to $(0,1)$, $(1,0)$, $(0,1)$, $(1,0)$
modulo $2$. $S_{n+1}$ follows $S_n$ with $B_n=A_{n+1}-(1,3)$.
\item Let an ordinal of level $n\ge3$ have vertex $A=(u,v)$. Its children are translates of
$S_k$, $1\le k\le n-2$, with vertices
\begin{align*}
&A+\big(2^n-2^{k+2}+4,\ \ 2^n-2^{k+2}-2(n-k)+4\big) &&\text{(first half)},\\
&A+\big(2^n+2^{k+1}+2(n-k)-4,\ \ 2^n+2^{k+1}-4\big) &&\text{(second half)}.
\end{align*}
\item An ordinal of level $n\ge2$ with vertex $(u,v)$ contains the \emph{switching square}
$[u+2^n,\,u+2^n+2n-4]^2$ of side $2n-3$, coloured by the central lattice: a cell of the square is
red if and only if both its coordinates are even. Each side of the square contains $n-1$ red
and $n-2$ blue cells. The two diagonal cells $(u+2^n-2,u+2^n-2)$ and $(u+2^n+2n-2,u+2^n+2n-2)$
are red. For $S_n$ the switching square is $[3\cdot2^n-2n+2,\ 3\cdot2^n-2]^2$.
\item Outside the main ordinals, cells $(x,y)$ with $\min(x,y)\ge2$ have the colour
$\Lam_2$, except in the transition squares $T_n$ of Proposition~\ref{prop:transition}
and at $(2,2)$ and $(4,4)$.
Inside an ordinal of level $n\ge3$, exclude its children, its transition squares,
its switching square and the two red diagonal cells of (iv). The remaining cells
have the colour $\Lam_0$ before the switching square and $\Lam_2$ after the square.
\item The white cells of $t$ are the borders of the ordinals and the two white lines.
\end{romenum}
\end{thm}

\begin{proof}
All parts restate Theorem~\ref{thm:main} in the coordinates $(x,y)$. For (ii),
$A_n=\big(\tfrac{e_n-w_n}2,\tfrac{e_n+w_n}2\big)$ and $B_n=A_n+(2^{n+1}-3)(1,1)$, and the
parities follow from Lemma~\ref{lem:timeline}(a), since $u=(e-w)/2$ is even, $v=u+w$ is odd and
$k=L/2$ is odd. For (iii), the vertex of a child opening at $e+p$ with half-width $w_k$ is
$A+\big(\tfrac p2+(n-k),\ \tfrac p2-(n-k)\big)$, and one inserts $p=p_{n,k}$ or $q_{n,k}$. For
(iv), the diamond of rule \ref{D4} has centre $\mu/2=u+2^n+n-2$ on the diagonal and half-side
$r/2=n-2$, and $u$ is even. The two red diagonal cells are the \ref{D2}-cells $\cl{\mu-2n}0$
and $\cl{\mu+2n}0$, which are red because $\mu\pm2n\equiv0\pmod4$. For (v), the cells in
question are given by rules \ref{D2}, \ref{D3} and \ref{D5}, and Lemma~\ref{lem:transinsquares}
below places all \ref{D2}- and \ref{D3}-cells in transition squares, apart from the cells
named. Part (vi) is \ref{P1}.
\end{proof}

\begin{cor}[the mirrored ordinals]\label{cor:ordinals}
For each $n\ge1$, the boxes contained in $S_n$, including $S_n$ itself, form an ordered
rooted tree under the child relation. This tree is isomorphic to the occurrence tree of
$s_n$: the root corresponds to the whole word and each child to its particular occurrence
in the defining list. The family $\cB$ is the disjoint union of these trees over the main
boxes. In a box of level at least $3$, the switching square corresponds to the bar of
its composite word.
\end{cor}
\begin{proof}
By Definition~\ref{def:children} the children of a box of level $n\ge3$ are, in time order,
boxes of levels $n-2,\dots,1,1,\dots,n-2$, and boxes of level $1$ and $2$ have no children,
which is the definition of $s_n$ with $s_1=\emptyset_1$ and $s_2=\emptyset_2$. By
Lemma~\ref{lem:timeline}(c), for levels at least $3$, the switching square lies between the
two copies of $S_1$.
\end{proof}

\begin{cor}[number of boxes]\label{cor:two-count}
Let $N_n$ be the number of boxes contained in $S_n$, including $S_n$ itself and all its
descendants. Then
\[
N_n=\frac{2^n-(-1)^n}{3}\qquad(n\ge1).
\]
\end{cor}
\begin{proof}
The children give $N_1=N_2=1$ and $N_n=1+2\sum_{k=1}^{n-2}N_k$ for $n\ge3$.
Subtracting consecutive equations gives $N_n=N_{n-1}+2N_{n-2}$ for $n\ge3$.
The displayed expression has the same initial values and satisfies this recurrence.
\end{proof}

\subsection{Transition squares}

Consecutive edges of a host that are $4$ rows apart produce a square of cells with a simple
structure. A \emph{junction} is a pair of consecutive elements $(c_0,a)$, a closing, and
$(c_0+4,b)$, an opening, of some $\cE(H)$, that is, a case J2, J3 or J5. Its \emph{transition
square} is the set
\[
\{\cl sd:\ |s-s_J|+|d|\le R_J\},\qquad (s_J,R_J)=\begin{cases}(c_0+1,\ a+1)&\text{in case J2},\\
(c_0+3,\ a+3)&\text{in cases J3 and J5},\end{cases}
\]
which is a square with sides parallel to the axes in the coordinates $(x,y)$. For the junction
between $S_n$ and $S_{n+1}$ this is the square
\[
T_n=[\,2^{n+2}-2n-1,\ 2^{n+2}+1\,]^2
\]
of side $2n+3$, whose top left corner $(u_B,u_B)$ has the first coordinate of $B_n=(u_B,v_B)$.

\begin{lem}\label{lem:transinsquares}
Every transition cell lies in the transition square of some junction, with the following
exceptions: the diagonal cells $\cl{c_{\mathrm{gap}}+1}0$ and $\cl{e_{\mathrm{gap}}-1}0$ on the two sides of each
switching gap $(c_{\mathrm{gap}},e_{\mathrm{gap}})$, and the three cells $\cl20$, $\cl40$ and $\cl80$, which are
$(1,1)$, $(2,2)$ and $(4,4)$.
\end{lem}
\begin{proof}
For a junction $J$, let $X_0$ be the box closing at $c_0$ and let $Y_0$ be the box
opening at $c_0+4$. The \emph{closing chain} $X_0\supset X_1\supset\cdots$ follows
last children, so $X_{j+1}$ is the last child of $X_j$ as long as $X_j$ has level at
least $3$. The \emph{opening chain} $Y_0\supset Y_1\supset\cdots$ follows first
children in the same way. By Lemma~\ref{lem:timeline}(b), $X_j$ closes
at $c_0-4j$ with half-width $a-4j$, and $Y_j$ opens at $c_0+4+4j$ with half-width $b-4j$.

Every closing edge lies in a closing chain, with one exception. Given a box $X$, go up
through the relation \emph{last child of} until a box $Z$ is reached that is not a last child.
Then $Z$ is followed by a sibling, which is case J2 or J3, or $Z$ is a main box, which is case
J5, or $Z$ is the first-half copy of $S_1$ before a switching gap, in which case $Z=X$ because
$S_1$ has no children. Symmetrically, every opening edge lies in an opening chain, except the
opening of the copy of $S_1$ that follows a switching gap and the opening of the main box $S_1$.

The transition rows of the chain edges lie in the square. In case J2, with $s_J=c_0+1$ and
$R_J=a+1$, the value of $|s-s_J|+|d|$ at the outermost cells of the transition rows is as
follows. For the closing of $X_j$: $a+1$ on the row $c_0-4j-1$ with $|d|\le a-4j-1$, $a+1$ on
the row $c_0-4j-2$ with $|d|\le a-4j-2$, and at most $a-1$ on the rows $c_0-4j+1$ and
$c_0-4j+2$. For the opening of $Y_j$: $a-1$ and $a-3$ on the rows $c_0+3+4j$ and $c_0+2+4j$,
and $a+1$ on the rows $c_0+5+4j$ with $|d|\le a-3-4j$ and $c_0+6+4j$ with $|d|\le a-4-4j$. In
cases J3 and J5, with $s_J=c_0+3$, $R_J=a+3$ and $b=a+2$, the same computation gives $a+3$ and
$a+3$ before the closing of $X_j$, $a+1$ and at most $a$ after it, $a+1$ and at most $a+1$
before the opening of $Y_j$, and $a+3$ and $a+3$ after it. All these values are at most $R_J$.

The edges outside these chains account for the exceptions in the statement.
Before a switching gap, the closing of the first-half copy of $S_1$ contributes
$\cl{c_{\mathrm{gap}}+1}0$. After the gap, the opening of the second copy contributes $\cl{e_{\mathrm{gap}}-1}0$.
The opening of the main box $S_1$ contributes $\cl80$, and the virtual edge
contributes $\cl20$ and $\cl40$. The interior
cell of each such copy of $S_1$ is also a transition cell of the other edge of the copy, which
lies in a chain.
\end{proof}

\begin{prop}[transition squares]\label{prop:transition}
Let $n\ge1$.
\begin{alphenum}
\item In $T_n$, every odd anti-diagonal is entirely white or entirely red, and white and red
alternate. The white anti-diagonals are the closing edges of $S_n\supset S_{n-2}\supset\cdots$
and the opening edges of $S_{n+1}\supset S_{n-1}\supset\cdots$, each as wide as $T_n$ at its
row. Every even anti-diagonal of $T_n$ is coloured by the central lattice.
\item Consequently, the four corner cells of $T_n$ are blue. Along the two sides that meet at
the top left corner, read from that corner, the colours repeat with period $4$ as blue, red,
blue, white for $n$ even and as blue, white, blue, red for $n$ odd. Along the two other sides,
read towards the bottom right corner, the colours repeat as blue, white, blue, red. The white
cells lie on the diagonal segments joining corresponding cells of the two initial sides, and
of the two final sides. The red cells form staircases joining corresponding red cells of the
sides, with runs of $2,3,3,\dots,3,2$ cells. All other cells are blue.
\end{alphenum}
\end{prop}

\begin{proof}
(a) Let $c_0=e_n+L_n$. The rows $c_0-4j$ with $j\ge0$ are the closing edges of the chain of last
children, of half-width $w_n-4j$, which is the width of $T_n$ at that row. The rows $c_0+4+4j$
are the opening edges of the chain of first children of $S_{n+1}$, of half-width $w_n+2-4j$,
again the width of $T_n$. The rows $c_0+2-4j$ and $c_0+6+4j$ are red on their full width in $T_n$
by rule \ref{D3} of the two neighbouring edges, except at the two outermost cells
$\cl{c_0+2}{\pm(w_n+2)}$, which are environment cells with $\Lam_2=1$. On even rows, $D=\Gam$ by
rule \ref{D2}, or $D=2=\Gam$ at the cells that rule \ref{D2} does not reach, which are blue by
\ref{P2} or rule \ref{D5} and have $\Gam=2$ by the congruences of Lemma~\ref{lem:timeline}(a).
The top and bottom corner cells lie in the interiors of the innermost boxes of the two chains,
where $D=\Gam=2$ by class (E1) and \ref{P3}.

(b) In the coordinates $(x,y)$ a red anti-diagonal is joined to the red cells of the central
lattice on the two neighbouring even anti-diagonals. At a cell $(x,y)$ with $x$ even the path
turns through $(x,y-1)$, and with $x$ odd through $(x+1,y)$, which produces the staircases
with runs of $3$ cells, and of $2$ cells at the two ends. The colours along the sides follow
from (a) and from the congruences $c_0\equiv w_n+2$ and $e_{n+1}\equiv w_n+2\pmod4$.
\end{proof}

The same structure occurs at every junction inside a box. Suppose a copy of $S_k$
with closing vertex $(u_B,v_B)$ is followed $4$ rows later by a sibling.
If the sibling is a copy of $S_{k+1}$, the transition square is
$[u_B,\,u_B+2k+2]^2$, of side $2k+3$. If the sibling is a copy of $S_{k-1}$,
the square is $[u_B,\,u_B+2k]^2$, of side $2k+1$. The proof of (a) applies with
the chains of the junction in place of the chains of $S_n$ and $S_{n+1}$.

\section{The argument triples}\label{sec:triples}

\begin{thm}\label{thm:triples}
Let $\Theta=\{(t(x-1,y),t(x-1,y-1),t(x,y-1)) : x,y\ge1\}$ be the set of triples at which $f$ is
evaluated in the computation of $t$.
\begin{alphenum}
\item $\Theta=\F^3\setminus\{\trip000,\trip101,\trip012,\trip210\}$, so $\Theta$ has $23$ elements.
\item For $g:\F^3\to\F$, the double sequence with rule $g$ and the same boundary values equals
$t$ if and only if $g=f$ on $\Theta$. Consequently $3^4=81$ functions generate $t$, of which
$3^3=27$ are symmetric in $x$ and $z$. Agreement with $f$ on a proper subset of
$\Theta$ does not force equality of the sequences.
\item The triples $\trip001$ and $\trip100$ occur at $(2,3)$ and $(3,2)$ only. The triples
$\trip002$ and $\trip200$ occur at $(2,y)$ with $y\ge4$ and at $(x,2)$ with $x\ge4$ only. The
triples $\trip110$ and $\trip011$ occur at $(1,y)$ with $y\ge3$ and at $(x,1)$ with $x\ge3$
only. Each of the remaining $17$ triples occurs at infinitely many cells with $x,y\ge3$.
\end{alphenum}
\end{thm}

\begin{proof}
(a) We first show that the four triples do not occur. By \ref{P1} a white cell is a border
cell or a cell of a white line.

The triples $\trip000$ and $\trip101$ require $\cy$ to be white. If $\cy$ is on a white line,
then $c=\cl s{\pm(s-4)}$ with $s\ge5$, a cell of class (E6), and the triple is $\trip00z$ or
$\trip z00$ with $z=D\cl{s-1}{\pm(s-5)}\ne0$, namely $z=1$ for $s=5$ and $z=2$ for $s\ge6$. If
$\cy$ is a border cell, its row is odd, so $\cx$ and $\cz$ lie on an even row and are not white.
If $\cy$ is a wall cell or the corner of a closing edge, then $\cx$ or $\cz$ is blue by
\ref{P2}, and if $\cy$ is an inner cell of an edge, then $\cx$ and $\cz$ are consecutive
$\Gam$-cells, one red and one blue.

The triple $\trip012$ requires $\cx$ white, $\cy$ red and $\cz$ blue. If $\cx$ is on a white
line, then $d+1=\pm(s-3)$. For $d+1=s-3$ we have $c=\cl s{s-4}$, and $\cy$ is on a white line
for $s\ge5$ while $\cy=\cl20$ is blue for $s=4$. For $d+1=-(s-3)$ the cell $c$ is on a white
line and $\cz=\cl{s-1}{-(s-1)}$ is red. If $\cx$ is a wall cell at $d+1=w$, then $\cy$ is blue by
\ref{P2}, or $\cz$ is white, at $s=e+1$ or $s=e+L+1$. If $\cx$ is a wall cell at $d+1=-w$, then
$\cy$ is blue by \ref{P2}. If $\cx$ is an inner cell of an edge, then $\cz$ is white as well. By
symmetry $\trip210$ does not occur either.

The other $23$ triples occur. Table~\ref{tab:triples} lists a first occurrence of each.

(b) If $g=f$ on $\Theta$, induction on $x+y$ shows that the two sequences agree. Conversely,
suppose that $g$ differs from $f$ on a nonempty set $\Theta'\subseteq\Theta$ and let $s^\ast$
be the least anti-diagonal on which some triple of $\Theta'$ occurs in $t$. On the
anti-diagonals before $s^\ast$ the two sequences agree, and at an occurrence of a triple of
$\Theta'$ on $s^\ast$ they differ. The same argument shows that agreement on a proper subset of $\Theta$ does not
force equality. Among the $81$ functions that agree with $f$ on $\Theta$, a function is
symmetric if and only if it takes the same value at $\trip012$ and $\trip210$, which leaves
$27$ choices.

(c) By the proof of (a), $\cx=\cy=0$ forces $\cy$ onto a white line with $c=\cl s{s-4}$ and
$s\ge5$, that is $c=(2,s-2)$, which gives $\trip001$ for $s=5$ and $\trip002$ for $s\ge6$. The
mirror statement gives $\trip100$ and $\trip200$. For $\trip110$, the parent $\cz$ is white and
$\cx,\cy$ are red. If $\cz$ is a border cell, the analysis of $\trip012$ with $x$ and $z$
exchanged shows that $\cy$ is blue or $\cx$ is white. If $\cz$ is on a white line with
$d-1=-(s-3)$, then $c=\cl s{-(s-4)}$ and $\cy$ is white or blue. Hence $d-1=s-3$, so $c$ is on
the white line $x=1$, where the triple is $\trip110$ for $s\ge4$. The mirror statement gives
$\trip011$. Each of the remaining $17$ triples occurs at a cell of the kind named in the last
column of Table~\ref{tab:triples}. Cells of each listed kind occur inside every main
box $S_n$ with $n\ge5$: at the openings, closings or walls of its children, in its
switching square, or in its environment.
\end{proof}

\begin{table}[ht]
\centering\small
\begin{tabular}{ccc>{\raggedright\arraybackslash}p{5.6cm}}
\toprule
$(x,y,z)$ & $f$ & first occurrence $(x,y)$ & where it occurs\\
\midrule
$\trip001$ & 2 & $(2,3)$ & only there\\
$\trip100$ & 2 & $(3,2)$ & only there\\
$\trip002$ & 2 & $(2,4)$ & $x=2$, $y\ge4$\\
$\trip200$ & 2 & $(4,2)$ & $y=2$, $x\ge4$\\
$\trip110$ & 0 & $(1,3)$ & white line $x=1$\\
$\trip011$ & 0 & $(3,1)$ & white line $y=1$\\
$\trip111$ & 2 & $(1,1)$, $(7,7)$ & \ref{D2}-cells before an edge, classes (E1), (E4)\\
$\trip112$ & 0 & $(1,2)$, $(5,4)$ & edges\\
$\trip211$ & 0 & $(2,1)$, $(4,5)$ & edges\\
$\trip010$ & 2 & $(5,5)$ & \ref{D2}-cells after an edge with $\cy$ red\\
$\trip020$ & 1 & $(2,2)$, $(6,6)$ & \ref{D2}-cells after an edge with $\cy$ blue\\
$\trip121$ & 1 & $(4,4)$ & outermost \ref{D2}-cells before an opening\\
$\trip102$ & 1 & $(7,6)$ & \ref{D3}-cells after an edge, class (E3)\\
$\trip201$ & 1 & $(6,7)$ & mirror image of the previous line\\
$\trip202$ & 0 & $(5,6)$ & walls\\
$\trip021$ & 2 & $(6,4)$ & class (E2), \ref{D2}-cells before a closing\\
$\trip120$ & 2 & $(4,6)$ & mirror image of the previous line\\
$\trip022$ & 2 & $(7,5)$ & class (E2)\\
$\trip220$ & 2 & $(5,7)$ & mirror image of the previous line\\
$\trip122$ & 2 & $(4,8)$ & lattices\\
$\trip221$ & 2 & $(3,5)$ & lattices\\
$\trip212$ & 2 & $(3,3)$ & lattices\\
$\trip222$ & 1 & $(3,4)$ & lattices, \ref{D3}-cells in case J4 and in boxes of level $2$\\
\midrule
$\trip000,\trip101,\trip012,\trip210$ & & never & \\
\bottomrule
\end{tabular}
\caption{The $23$ triples $(t(x-1,y),t(x-1,y-1),t(x,y-1))$ that occur in the computation of
$t$, the value of $f$ at each, a cell $(x,y)$ at which the triple first occurs, and the kind of
cell at which it occurs. Where two cells are given, the second is the first occurrence with
$x,y\ge3$.}
\label{tab:triples}
\end{table}

\clearpage
\section{Three letters and three initial symbols}\label{sec:three-initial}

The construction with two initial symbols has the following analogue with three. Put
\[
s_1=\emptyset_1,\qquad s_2=\emptyset_2,\qquad s_3=\emptyset_3,\qquad
s_n=(s_{n-3}\cdots s_1\mid s_1\cdots s_{n-3})\quad(n\ge4).
\]
The initial symbols are distinct atoms. We give a three-letter sequence whose rectangles,
now bordered in red, realise these words.

Let $\rho:\N^2\to\F$ have boundary values $\rho(x,0)=\rho(0,y)=0$ and recurrence rule
\begin{equation}\label{eq:three-initial-rule}
h(x,y,z)=(y-1)^2(x^2+z^2)+y(y+1)(x+z)+xz(x+z+2xz)+2y+1
\quad\text{in }\F.
\end{equation}
As before, $0$ is white, $1$ red and $2$ blue. Figure~\ref{fig:three-initial} shows the
beginning of the sequence. For the proof it is convenient to exchange $0$ and $1$, fixing
$2$. Thus $u=1-\rho$ has boundary value $1$ and the rule
\begin{equation}\label{eq:three-initial-aux}
h_0(x,y,z)=2+2y+2y^2+2y^2(x^2+z^2)+2xz(x+z)+x^2z^2.
\end{equation}
Indeed $h(x,y,z)=1-h_0(1-x,1-y,1-z)$ in $\F$. All pattern values below refer to $u$.
Subtracting these values from $1$ in $\F$ gives the colours of $\rho$.

\begin{figure}[ht]
\centering
\includegraphics[width=0.48\textwidth]{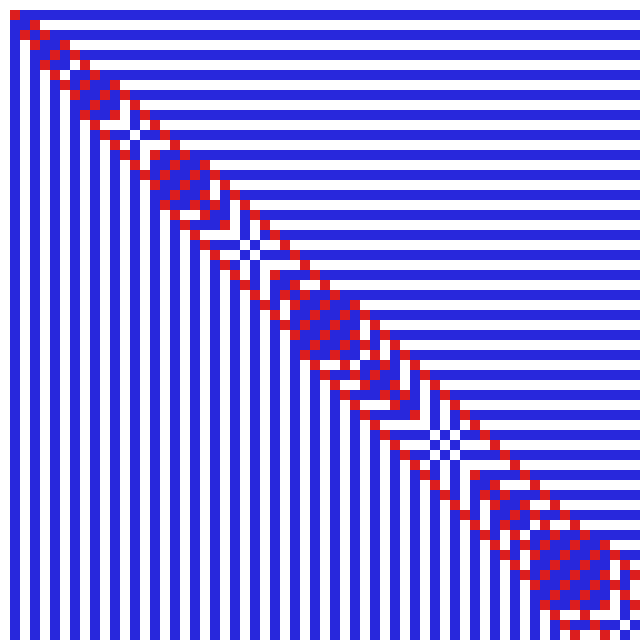}\hfill
\includegraphics[width=0.48\textwidth]{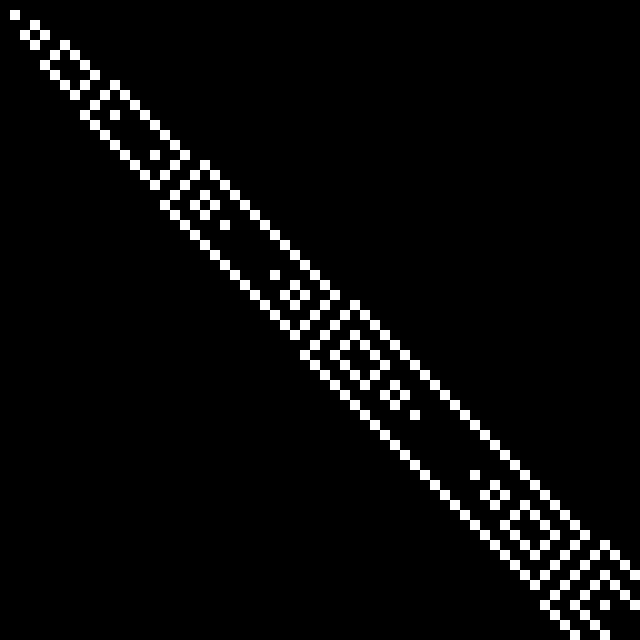}
\caption{The cells $(x,y)$ with $0\le x,y<64$ of $\rho$.
Left: white $0$, red $1$ and blue $2$. Right: red is drawn in white, and the other two
colours in black. The first five main boxes are complete, and the sixth begins at the
bottom right. The first three boxes have no children. The fourth contains two copies of
the first, and the fifth contains copies of the second, first, first and second.}
\label{fig:three-initial}
\end{figure}

\subsection{The box patterns}
Define their durations by
\[
b_1=2,\quad b_2=6,\quad b_3=14,\qquad
b_n=b_{n-1}+2b_{n-3}+8\quad(n\ge4),
\]
and their main opening times by $c_1=5$, $c_{n+1}=c_n+b_n+3$. The box region of level $n$
opening at $e$ is
\[
C(n,e)=\{\cl sd:e\le s\le e+b_n,\ |d|\le n\},\qquad e\ge n,\quad e\equiv n\pmod2.
\]
The main regions are $C_n=C(n,c_n)$. A region of level $n\le3$ has no children. For $n\ge4$
its children are $C(k,e+\alpha_{n,k})$ and $C(k,e+\beta_{n,k})$, $1\le k\le n-3$, where
\[
\alpha_{n,k}=3+\sum_{j=k+1}^{n-3}(b_j+3),\qquad
\beta_{n,k}=b_n-b_k-\alpha_{n,k}.
\]
Write $\mathcal H$ for the main regions and all their descendants.

The first child opens $3$ rows after its parent, consecutive children within each half
are separated by $3$ rows, and the last child closes $3$ rows before its parent. The gap
between the two children of level $1$ has length $2n+8$. These statements follow from
\begin{equation}\label{eq:three-initial-length}
b_n=2\sum_{k=1}^{n-3}b_k+8n-10\quad(n\ge3).
\end{equation}
The sum is empty for $n=3$. Induction also gives
\[
b_n\equiv2\pmod4,\quad c_n\equiv n\pmod4,\quad
\alpha_{n,k}\equiv n-k,\quad\beta_{n,k}\equiv k-n\pmod4,
\]
and $b_n/2\equiv2n+1\pmod4$ for $n\ge3$.

\begin{defn}[the pattern $H_n$]\label{def:three-initial-pattern}
Let $0\le r\le b_n$, $|d|\le n$ and $n+r\equiv d\pmod2$.
The value $H_n(r,d)$ is given by the first applicable rule below. An edge $(a,k)$ means
$(0,n)$ or $(b_n,n)$, or an opening or closing edge of a child, at relative time $a$
and with half-width $k$.
\begin{romenum}
\item On the border, $H_n(r,d)=0$.
\item In a child of level $k$ opening at offset $a$, use $H_k(r-a,d)$.
\item If $1\le|r-a|\le2$ and $|d|\le k-|r-a|$ for an edge $(a,k)$, use $2$.
\item If $n\ge3$, $d=0$ and $|r-b_n/2|=n+1$, use $0$.
\item If $n\ge3$ and $|r-b_n/2|+|d|\le n-1$, use $1$ for even $d$ and $2$ for odd $d$.
\item Otherwise use $1$ if
$\min(r,b_n-r)+|d|\equiv n-2\pmod4$, and $2$ otherwise.
\end{romenum}
\end{defn}

Rules (iv) and (v) describe the two exceptional diagonal cells and the switching diamond.
For $n=1,2$ these rules do not apply. In particular the first three patterns, although
they contain no child boxes, are distinct.

Define $U$ on the whole quadrant as follows. Set $U\cl20=0$. Inside $C_n$ use
$H_n(s-c_n,d)$. Outside the main regions, put $U\cl sd=2$ if
$1\le|s-a|\le2$ and $|d|\le n-|s-a|$ for a main edge $(a,n)$. At every remaining cell put
\[
U\cl sd=\begin{cases}1,&s-|d|\equiv0\pmod4,\\2,&\text{otherwise.}\end{cases}
\]
In particular the axes have value $1$.

\begin{lem}[local verification]\label{lem:three-initial-local}
The function $U$ satisfies the recurrence with rule $h_0$.
\end{lem}
\begin{proof}
For a centre $\cl sd$ with $|d|\le s-2$, its ordered window consists of the values at
\[
\cl sd,\quad \cl{s-1}{d+1},\quad \cl{s-2}d,\quad \cl{s-1}{d-1}.
\]
We reduce these windows to a finite catalogue and check the catalogue by exact
evaluation of Definition~\ref{def:three-initial-pattern}. Reflection exchanges the
two side parents, and $h_0$ is symmetric in those arguments, so it suffices to use
centres with $d\ge0$. The reduction has three parts: windows meeting a central
diamond or exceptional point, windows meeting a border or transition row, and
windows following a single lattice rule.

\emph{The five contexts of a box.} A level-$k$ box is a main box, a child in the
first half of its parent's list, or a child in the second half. In the first half
we distinguish the first child, and in the second half the last child. These give
the following representatives, all taken from the explicit description $U$:
\[
\begin{array}{l|l}
\text{context of the level-$k$ box}&\text{opening time of its representative}\\ \hline
\text{main box}&c_k\\
\text{first child in the first half}&c_{k+3}+3\\
\text{other child in the first half}&c_{k+4}+b_{k+1}+6\\
\text{last child in the second half}&c_{k+3}+b_{k+3}-b_k-3\\
\text{other child in the second half}&c_{k+4}+b_{k+4}-b_k-b_{k+1}-6.
\end{array}
\]
These representatives also have the required lattice outside the box. If $r$ is
measured from the child's opening, substitution of $\alpha_{n,k}\equiv n-k$ in
the parent's first-half lattice gives value $1$ when $r+|d|\equiv k-2\pmod4$.
Substitution of $\beta_{n,k}\equiv k-n$ in its second-half lattice gives value $1$
when $b_k-r+|d|\equiv k-2\pmod4$. The main exterior has the latter phase because
$c_k\equiv k\pmod4$. Thus the outside phase depends on the listed context, not
on the level of a more distant ancestor.

\emph{Windows at the centre of a box.} Put $h=b_n/2$, $u=r-h$ and $a=|d|$.
Suppose a window meets the diamond $|u|+a\le n-1$ or one of the points
$d=0$, $|u|=n+1$, where $n\ge3$. On all four cells of that window the values are
given, in the displayed order of priority, by
\[
J_n(u,d)=\begin{cases}
0,&a=n,\\
2,&a=0,\quad |u|=n+3,\\
0,&a=0,\quad |u|=n+1,\\
1,&|u|+a\le n-1,\quad d\text{ even},\\
2,&|u|+a\le n-1,\quad d\text{ odd},\\
1,&|u|-a\equiv n-1\pmod4,\\
2,&\text{otherwise.}
\end{cases}
\]
To verify this formula, the adjacent children of level $1$ close and open at
$u=-(n+4)$ and $u=n+4$. Their transition cells at $d=0$, $|u|=n+3$ have value $2$.
For $n=3$ the same cells are transitions of the box's own edges. No other child
reaches the window. The remaining interior cells follow rule (vi), where
$\min(r,b_n-r)=h-|u|$ and $h\equiv2n+1\pmod4$ give the stated congruence.
A window can also reach the wall $a=n$, of value $0$, and the exterior column
$a=n+1$. On this exterior column the two possible outside phases agree at the
permitted parity, and give the last two lines of the formula.

The possible nonnegative centres satisfy
\[
-n-1\le u\le n+3,\qquad 0\le d\le n,\qquad n+1+u\equiv d\pmod2,
\]
with at least one of the four cells in the diamond or at an exceptional point.
If all four cells lie in the diamond, the parent triple is $(2,1,2)$ or
$(1,2,1)$, giving the required value $1$ or $2$. For any other such window with
$n\ge16$, the following map lowers the level by $4$ and preserves all four values:
\[
(n,u,d)\longmapsto
\begin{cases}
(n-4,u,d-4),&d\ge6,\\
(n-4,u-4\operatorname{sgn}(u),d),&d\le5.
\end{cases}
\]
In the first case all four absolute positions decrease by $4$. Their distances
from the wall, the diamond inequality and the residues modulo $4$ are preserved,
and no cell reaches $d=0$. In the second case, a cell outside the diamond implies
$|u|\ge n-7\ge9$. All four time offsets therefore have the same sign and their
absolute values decrease by $4$. This preserves the diamond and exceptional-point
tests, including the two transition cells in $J_n$. All positions are at most
$6$, below the new wall at $n-4\ge12$. Iteration reduces the check to $3\le n\le15$.

\emph{Windows near horizontal edges.} Exclude the central windows just treated.
If an edge at time $e$ or one of its transition cells belongs to a window, its
centre satisfies $-2\le s-e\le4$. Hence the relevant edge times lie in
$[s-4,s+2]$. Choose the edge of greatest half-width $K$ in this interval.
Then $d\le K+1$. Outside the three small patterns, successive nearby edges are
separated by $3$ rows. The changes in half-width are $-3$ from a parent's opening
to its first child's opening, $-1$ between siblings in the first half, $1$ between
siblings in the second half, $3$ from the last child's closing to its parent's
closing, and $1$ between successive main boxes. An interval of length $6$ meets
at most three such edges, whose widths differ by at most $6$.

The context table specifies these neighbours. A first or last child has parent
level $K+3$, and a nonextreme child has the neighbouring sibling of level $K+1$
on the relevant side. The box's own first and last children have level $K-3$.
Further ancestors cannot contribute an edge in this interval without being
included in the choice of $K$. Their walls are at least three columns beyond
the selected wall, so they cannot reach the window. Their remaining contribution
is the outside lattice phase already calculated above. Consequently the five
contexts cover all horizontal-edge configurations.

For $K\ge20$, put $K'=16+((K-16)\bmod4)$. Use the same context at level $K'$,
preserving the centre's time offset from the selected opening or closing edge.
Replace its position by
\[
d'=\begin{cases}
d,&d\le1,\\
2+((d-2)\bmod4),&2\le d\le K-9,\\
d-(K-K'),&d\ge K-8.
\end{cases}
\]
All relevant levels decrease by $K-K'$, a multiple of $4$. Their edge spacings
remain $3$, and the endpoint phases remain unchanged because $b_k\equiv2\pmod4$.
In the last case, all four positions decrease by the same amount, preserving
wall distances and transition cutoffs. In the first two cases all four original
positions are at most $K-8$, within every relevant transition cutoff. The target
positions are at most $6$, while the smallest target cutoff is at least
$K'-8\ge8$. Horizontal rows therefore keep their values. Lattice values keep
their residues modulo $4$, including the two sides of $d=0$. Thus it suffices
to check levels $1\le k\le19$ in the five contexts, at each opening and closing,
with centre offsets $-2,\dots,4$ and $0\le d\le k+1$.

\emph{Windows near a wall.} After the central and horizontal-edge cases, a
window meeting a wall of half-width $k$ has centre $d\in\{k-1,k,k+1\}$. Its
inside and outside cells follow the two lattice phases determined by the context.
Children lie at least three columns further in and cannot meet this window.
A change between the two inside phases would meet the diamond, since the diamond
reaches $|d|=k-1$ on the middle row. Thus the only data needed are the side of
the middle row, the context, $s\bmod4$ and $d-k$.
For $k\ge6$, the same time residue can be represented by a relative row in
$6,\dots,9$ on the first side, or $b_k-9,\dots,b_k-6$ on the second side.
These rows avoid edges and the middle row. Reducing $k\ge20$ to $K'$ as above
and preserving $d-k$ reduces the width as well. For $k\le5$ we include all
relative rows $0,\dots,b_k$ in the catalogue.

\emph{The remaining windows and the finite check.} The initial cell $\cl20$
and windows containing it are checked directly, at centres $\cl20$, $\cl31$
and $\cl40$, together with their reflections. Any remaining window follows
one lattice phase throughout. Indeed, on the diagonal the children, borders
and transitions cover the noncentral parts of a box. In the central gap the
remaining diagonal cells are the two exceptional points and the diamond.
This follows from the gap $2n+8$, and is immediate in levels $1,2,3$ from
their definitions. The gaps of length $3$ between main boxes are covered by
their transitions. Thus no remaining window is centred on the diagonal.
Off the diagonal a single lattice phase has parent triple $(1,2,2)$ or $(2,2,1)$
when the value is $1$, and $(2,1,1)$ or $(1,1,2)$ when the value is $2$.
Substitution in \eqref{eq:three-initial-aux} verifies these four cases.

For the finite catalogue, discard centres of wrong parity and centres with
$d>s-2$. Evaluate the four entries from the explicit description, then check
that the first equals $h_0$ applied to the last three. This computation checks
$7930$ horizontal-edge windows, $1570$ wall windows and $858$ central windows,
with no failures.
These counts include repeated value quadruples in different contexts.
Table~\ref{tab:three-initial-local} summarises the resulting triples by cell type.
The reductions above establish coverage at arbitrary levels, so this finite
check proves the lemma.
\end{proof}

\begin{table}[ht]
\centering\small
\begin{tabular}{lcl}
\toprule
kind of cell & $U(c)$ & argument triples, up to reversal\\
\midrule
opening edge & $0$ & $222$\\
closing edge & $0$ & $102,\ 222$\\
wall, away from corners & $0$ & $102$\\
transition cell & $2$ & $011,\ 012,\ 020,\ 121,\ 202$\\
exceptional point before the diamond & $0$ & $222$\\
exceptional point after the diamond & $0$ & $111$\\
diamond, even position & $1$ & $022,\ 101,\ 122,\ 212$\\
diamond, odd position & $2$ & $011,\ 112,\ 121$\\
interior lattice & $1$ & $021,\ 022,\ 122$\\
interior lattice & $2$ & $011,\ 012,\ 112,\ 202$\\
exterior lattice & $1$ & $022,\ 122$\\
exterior lattice & $2$ & $011,\ 112$\\
initial cell $\cl20$ & $0$ & $111$\\
\bottomrule
\end{tabular}
\caption{The local check for $U$. Here $abc$ denotes the triple $(a,b,c)$, and reversal
means interchanging its first and third entries. Edge and wall cases refer to the
smallest box whose border contains the cell.}
\label{tab:three-initial-local}
\end{table}

\begin{thm}[three initial symbols]\label{thm:three-initial}
For all $x,y\ge0$, $\rho(x,y)=1-U\cl{x+y}{y-x}$ in $\F$.
The rectangles centred on the diagonal with red borders are exactly the boxes of
$\mathcal H$. Their main boxes are $C_1,C_2,\dots$, and the boxes directly inside $C_n$
have levels $n-3,\dots,1,1,\dots,n-3$ for $n\ge4$. The first three levels have no children.
Consequently their ordered nesting realises the mirrored construction with three initial symbols.
\end{thm}
\begin{proof}
Lemma~\ref{lem:three-initial-local} and uniqueness of the recurrence give $u=U$, and
the conjugacy of the rules gives $\rho=1-U$.

The zero cells of $U$ are precisely the borders in $\mathcal H$, the two exceptional
diagonal cells in each box of level at least $3$, and $\cl20$.
The latter cells have $d=0$, so none can lie on a wall of a rectangle of positive half-width.
An interior cell of an edge of half-width $n$ has value $2$ two rows before and after it
in its column. Hence a wall of a zero-bordered rectangle must follow a wall of a box of
the same half-width. Distinct such boxes have a nonzero gap in that column: their time
separation is at least $3$, and the transition and lattice rules give nonzero values
between their walls. A wall therefore belongs to a single box. Between its opening and
closing, no row of the appropriate parity is entirely zero: for width $n\ge2$ the
positions $|d|=n-2$ are nonzero, and the width-$1$ box has only its two border rows of
that parity. Thus the rectangle is the whole box. The child positions are those used
in its definition, so the asserted nesting follows.
\end{proof}

\begin{cor}[number of boxes]\label{cor:three-count}
Let $M_n$ count all boxes in $C_n$, including itself. Then
\[
M_1=M_2=M_3=1,\qquad M_n=M_{n-1}+2M_{n-3}\quad(n\ge4),
\]
and two representations of these numbers, valid for every $n\ge1$, are
\[
M_n=\sum_{j=0}^{\lfloor(n-1)/3\rfloor}2^j\binom{n-1-2j}{j},\qquad
M_n=\left\lfloor\frac{3^{n^2+2n}}{3^{3n}-3^{2n}-2}\right\rfloor\bmod 3^n.
\]
The second expression is an arithmetic term, using the least nonnegative remainder.
\end{cor}
\begin{proof}
The children give $M_n=1+2\sum_{k=1}^{n-3}M_k$ for $n\ge4$. Subtracting successive
equations gives the recurrence, including $n=4$ by the initial values. Its generating
function is $\sum_{n\ge1}M_nz^{n-1}=1/(1-z-2z^3)$. In the expansion
$\sum_{m\ge0}(z+2z^3)^m$, choosing $j$ factors $2z^3$ and $n-1-3j$ factors $z$
gives the displayed coefficient.

For the arithmetic term, we use the generating-function extraction method of
\cite[Section~4]{PS25}. Put $M_0=0$, so that
$G(z)=\sum_{n\ge0}M_nz^n=z/(1-z-2z^3)$.
Induction gives $M_r\le2^{r-2}$ for $r\ge2$, with strict inequality for $r\ge3$:
check $r=2,3,4$ directly and, for $r\ge5$, use
$M_r\le2^{r-3}+2\cdot2^{r-5}<2^{r-2}$.
In particular $G(3^{-n})$ converges for $n\ge1$. For an integer $A_n\ge0$,
\[
3^{n^2}G(3^{-n})=A_n3^n+M_n+\eta_n,\qquad
\eta_n=\sum_{i\ge1}\frac{M_{n+i}}{3^{ni}}.
\]
The tail satisfies
\[
0\le \eta_n<2^{n-2}\sum_{i\ge1}\left(\frac2{3^n}\right)^i
=\frac{2^{n-1}}{3^n-2}\le1,
\]
where the last inequality is $3^n\ge2^{n-1}+2$. Since $M_n<3^n$, taking the floor
and the remainder gives the asserted formula. Because its denominator is positive for
$n\ge1$, its subtractions can be read as truncated subtraction. The remainder can be
written as $a\bmod m=a-m\lfloor a/m\rfloor$. The expression is therefore an arithmetic
term in the sense of \cite{PS25}. The stronger coefficient bound above permits base $3$.
A direct application of \cite[Theorem~2]{PS25} gives the same formula with base $8$.
\end{proof}

\section{The construction with one initial symbol}\label{sec:four}

\subsection{The question}
We now ask for a recurrent double sequence over a small alphabet that realises
the construction with one initial symbol from Section~\ref{sec:constructions}:
\[
s_0=\emptyset,\qquad
s_{n+1}=(s_n\,s_{n-1}\cdots s_0\mid s_0\cdots s_{n-1}\,s_n).
\]
We seek white-bordered rectangles $S_n$ along the diagonal. The rectangles directly
inside $S_n$ must be translates of $S_{n-1},\dots,S_0,S_0,\dots,S_{n-1}$, in that order.
Here a rectangle is inside another when its closed region lies in the interior
of the enclosing rectangle.
As in Section~\ref{sec:coords}, we call these rectangles ordinals. For this construction
we impose one further condition: each row strictly between the opening and closing
edges must contain a non-white cell strictly between the walls. This condition excludes
rectangles drawn inside a white region and the union of two rectangles that share an edge.

\subsection{Three letters}
\begin{thm}\label{thm:three-none}
Let $g:\F^3\to\F$ be symmetric in its first and third argument and let $b\in\F$. The
recurrent double sequence with rule $g$ and boundary values $t(x,0)=t(0,y)=b$
cannot satisfy all of the following conditions.
\begin{romenum}
\item Along the diagonal, three consecutive ordinals contained in no other ordinal
are $S_0$, $S_1$ and $S_2$. There is no ordinal strictly inside $S_0$. The ordinals directly
inside $S_1$ are two copies of $S_0$, and those directly inside $S_2$ are copies of
$S_1,S_0,S_0,S_1$, in that order.
\item All three ordinals lie in the square $x,y<192$ and have half-width at most $40$.
\item The ordinal $S_1$ lies in the square $x,y<40$ and has half-width at most $9$.
\item No colour occupies more than $97$ percent of the cells in the square $x,y<40$.
\end{romenum}
\end{thm}
\begin{proof}
We use a finite condition that depends only on the contents of $S_2$. If (i)--(iii)
hold, the small square contains a nested pair, and the large square contains a box
whose direct children have levels $1,0,0,1$, each level-$1$ child having two leaf
children. The test allows this box to have an enclosing ordinal. It also allows
other boxes between the stipulated main boxes. Thus it makes no inference about
global maximality or global consecutiveness from the finite picture.

For the enumeration, write a rectangle as $(e,w,f)$, where $e$ and $f$ are its
opening and closing times. Its border and its interior rows are entirely visible
in the square $x,y<N$ if
\[
1\le w\le W,\qquad w\le e<f,\qquad e\equiv f\equiv w\pmod2,
\qquad f+w\le2N-2.
\]
An ordinal inside such a rectangle has smaller half-width and lies in the same
square, so every descendant needed to test its children is visible as well.
In coordinates, $(e,w,f)$ is inside $(e',w',f')$ if
$e'<e$, $f<f'$ and $w<w'$.

For each possible pair $(e,w)$, the detector checks the opening edge and then
extends the walls in steps of two rows. Each extension checks the newly enclosed
rows for a non-border-coloured cell. It stops at a broken wall, an entirely
border-coloured interior row, the boundary of the square, or the first closing
edge. The last stopping rule is valid because that closing edge would be an
entirely border-coloured interior row of any longer candidate with the same
$(e,w)$. Thus there is at most one ordinal for each opening and half-width.
For fixed $w$ there are at most $N-1-w$ possible openings, so the storage bounds are
\[
\sum_{w=1}^{9}(39-w)=306,\qquad
\sum_{w=1}^{40}(191-w)=6820.
\]
The program allocates these bounds and aborts if either is exceeded.

The test of descendants does not assign each rectangle an arbitrary parent.
For each candidate $B$, it lists all rectangles strictly inside $B$. A leaf has
none. A level-$1$ candidate has two descendants, both leaves. For a level-$2$
candidate, take the containment-maximal rectangles inside $B$ and require four
of them in the order $1,0,0,1$. The outer two must be level-$1$ candidates.
There can be at most eight proper descendants in this configuration. Checking
all candidates with four to eight descendants therefore includes the required
$S_2$, even if some containment relations elsewhere do not form a tree.
The test does not require copies to have matching dimensions or colours in their
interiors, so it accepts every configuration stipulated by (i).

There are $3^{18}$ symmetric rules. If $g(b,b,b)=b$, the whole sequence is
constant and fails (iv). Otherwise there is a unique permutation of the three
colours taking $b$ to $0$ and $g(b,b,b)$ to $1$. Conjugating $g$ by this
permutation preserves symmetry, containment and the density condition.
It is therefore enough to enumerate the $3^{17}=129\,140\,163$ symmetric rules
with boundary $0$ and $g(0,0,0)=1$, testing all three possible border colours.

The exhaustive run tested all these rules. All passed the density condition,
and $649\,981$ pairs of a rule and a border colour passed the small-window nesting test.
For each surviving pair, all ordinals of half-width at most $40$ in the large
square were enumerated. None contained a level-$2$ candidate. The largest
rectangle counts were $170$ in the small square and $3377$ in the large square,
below the proved storage bounds.
As a positive control, the same detector accepts the level-$2$ box of the
four-letter sequence at times $55$ to $105$, with its eight descendants.
Section~\ref{sec:computer} records the independent checks of the detector.
\end{proof}

Theorem~\ref{thm:three-none} rules out one geometric form within a finite window and for
symmetric rules with constant boundary values. It does not rule out sequences over three
letters with a non-symmetric rule, with periodic boundary values, with ordinals that first
appear beyond the window, or with examples excluded by the density or early-nesting conditions.

\subsection{Four letters: the rule}
Let $A=\{0,1,2,3\}$, let $g:A^3\to A$ be the function of Table~\ref{tab:g}, which is symmetric
in its first and third argument, and let $\tau:\N^2\to A$ be the recurrent double sequence with
$\tau(x,0)=\tau(0,y)=1$ and $\tau(x,y)=g(\tau(x-1,y),\tau(x-1,y-1),\tau(x,y-1))$. The colours are
$0$ white, $1$ red, $2$ blue and $3$ green. Figure~\ref{fig:four} shows the first
$192\times192$ cells in colour and in monochrome.

\begin{table}[ht]
\centering
\begin{tabular}{c|cccc|cccc|cccc|cccc}
\toprule
 & \multicolumn{4}{c|}{$x=0$} & \multicolumn{4}{c|}{$x=1$} & \multicolumn{4}{c|}{$x=2$} & \multicolumn{4}{c}{$x=3$}\\
$z=$ & 0&1&2&3 & 0&1&2&3 & 0&1&2&3 & 0&1&2&3\\
\midrule
$y=0$ & 1&0&2&3 & 0&1&0&1 & 2&0&0&0 & 3&1&0&2\\
$y=1$ & 0&2&2&3 & 2&0&1&0 & 2&1&2&0 & 3&0&0&0\\
$y=2$ & 1&2&3&1 & 2&3&0&0 & 3&0&0&3 & 1&0&3&3\\
$y=3$ & 1&3&0&2 & 3&0&1&0 & 0&1&1&2 & 2&0&2&0\\
\bottomrule
\end{tabular}
\caption{The values $g(x,y,z)$ of the rule over four letters. The eight triples
$(1,1,3)$, $(1,2,2)$, $(1,2,3)$, $(1,3,1)$, $(2,2,1)$, $(3,1,1)$, $(3,1,3)$, $(3,2,1)$ do not occur as
arguments of $g$ in the computation of $\tau$ (Corollary~\ref{cor:four-triples}), so $g$ may be
changed there without changing $\tau$.}
\label{tab:g}
\end{table}

\begin{figure}[ht]
\centering
\includegraphics[width=0.48\textwidth]{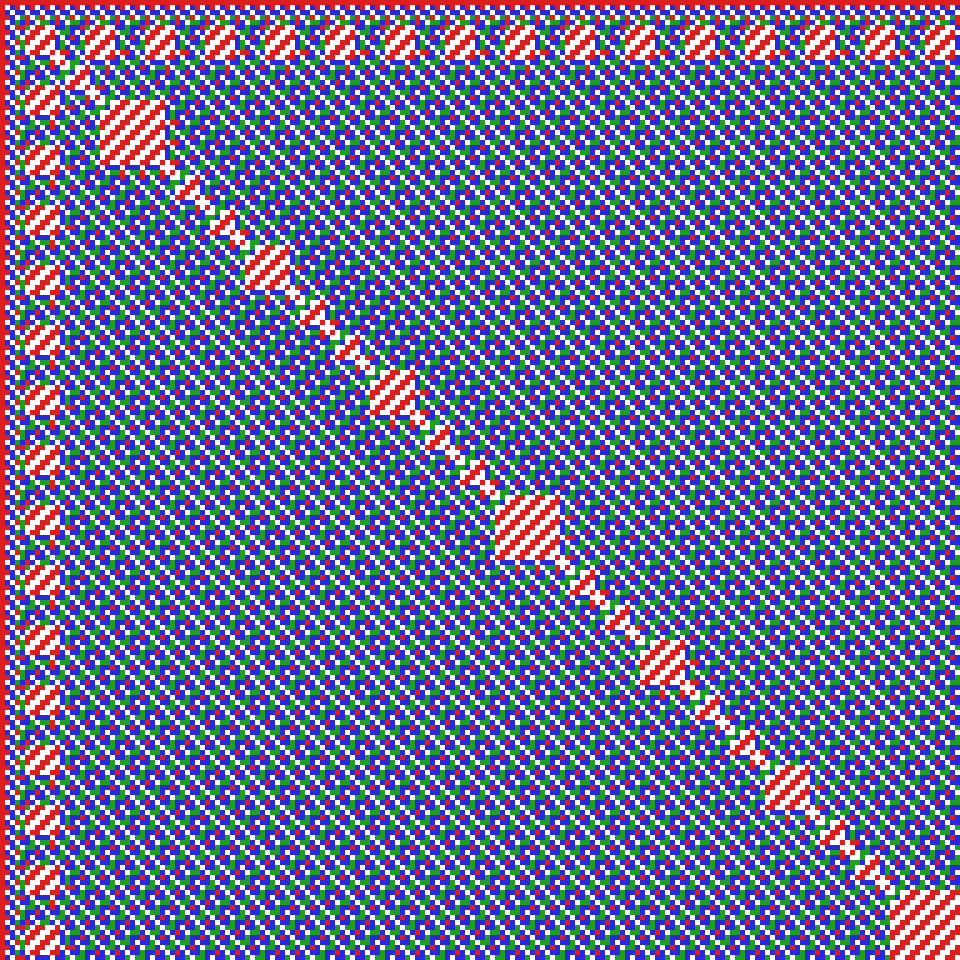}\hfill
\includegraphics[width=0.48\textwidth]{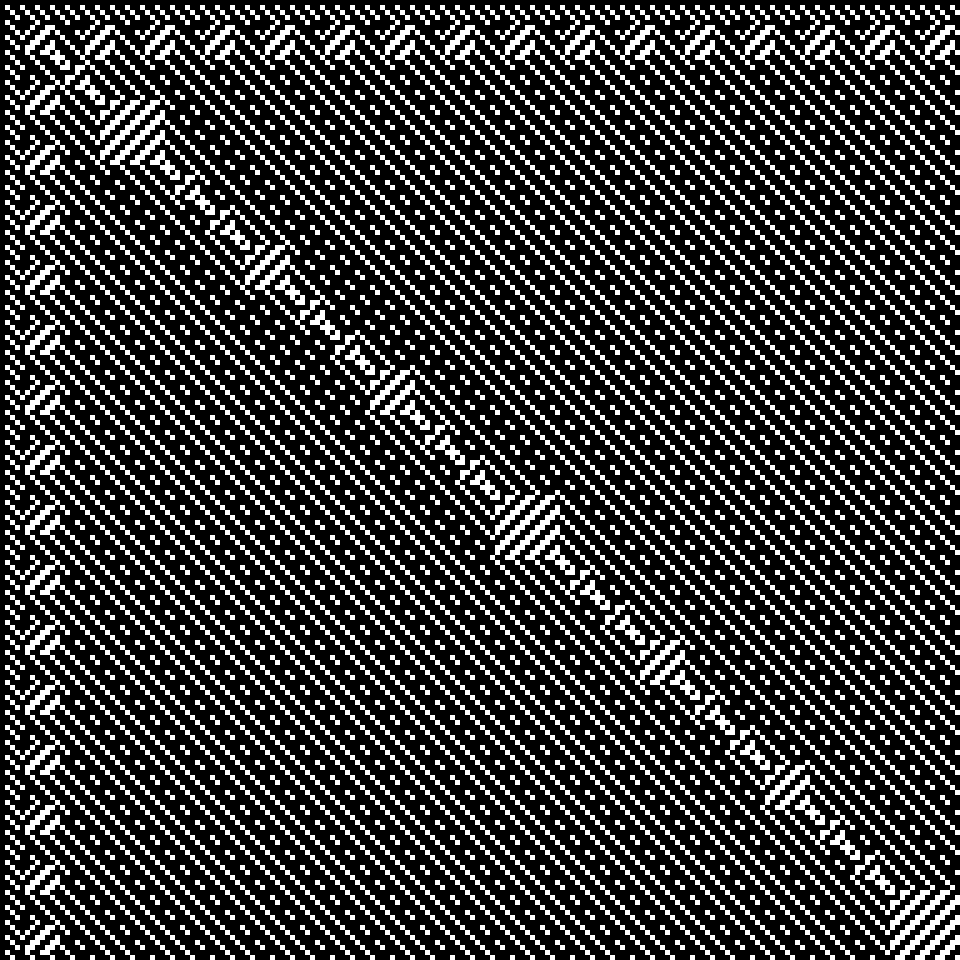}
\caption{The cells $(x,y)$ with $0\le x,y<192$ of the four-letter sequence $\tau$.
Left: white $0$, red $1$, blue $2$ and green $3$. Right: the same cells, with every non-white
colour drawn in black. Along the diagonal the complete main boxes have levels
$1,1,2,1,0,0,1,2,3$. The opening of the next box, of level $4$, enters at the bottom right.
The ordinary boxes have closed white borders. The dashed walls of the virtual regions
remain interrupted in the monochrome view, as described in Propositions~\ref{prop:triple}
and~\ref{prop:diagonal}.}
\label{fig:four}
\end{figure}

\subsection{Box regions and their children}\label{sec:four-boxes}
We use the rotated coordinates of Section~\ref{sec:rot}. For $n\ge0$ put
\[
v_n=4n+1,\qquad \ell_n=6\cdot3^n-4,
\]
so that $\ell_0=2$, $\ell_1=14$, $\ell_2=50$, $\ell_{n+1}=3\ell_n+8$ and
$\ell_n\equiv2\pmod6$. We first specify the regions and their positions. Their ordinary and
virtual colourings are defined in Definition~\ref{def:four-virtual}.

\begin{defn}[box region]\label{def:four-region}
For $n\ge0$ and odd $e\ge v_n$, the \emph{box region of level $n$ opening at $e$} is
\[
R_n(e)=\{\cl sd:e\le s\le e+\ell_n,\ |d|\le v_n\}.
\]
Its opening and closing edges are the rows $s=e$ and $s=e+\ell_n$ within the region.
Its walls are the columns $d=\pm v_n$, and its interior consists of the cells with
$e<s<e+\ell_n$ and $|d|<v_n$. Its half-width is $v_n$ and its vertex is
$((e-v_n)/2,(e+v_n)/2)$. This definition specifies a set of cells, before any colouring is assigned.
\end{defn}

\begin{defn}[child regions]\label{def:four-children}
A region of level $0$ has no children. For $n\ge1$, the children of $R_n(e)$ are
\[
R_k(e+a_{n,k}),\quad R_k(e+b_{n,k})\qquad(0\le k<n),
\]
where
\[
a_{n,k}=3^{n+1}-3^{k+2}+4,\qquad
b_{n,k}=\ell_n-\ell_k-a_{n,k}=3^{n+1}+3^{k+1}-4.
\]
In time order their levels are $n-1,n-2,\dots,0,0,\dots,n-2,n-1$.
\end{defn}

The first child opens $4$ rows after its parent. Consecutive children are separated by $4$
rows, except for the two children of level $0$, which are separated by $2$ rows.
The last child closes $4$ rows before its parent. These facts follow by subtracting the
offsets in Definition~\ref{def:four-children}. The offsets in the first half are congruent
to $4$ modulo $6$, and those in the second half to $2$ modulo $6$.

The identity $\ell_n=4+3\ell_{n-1}+4$ also places three regions of level $n-1$ at the
offsets $4$, $4+\ell_{n-1}$ and $4+2\ell_{n-1}$. Adjacent regions share an edge.
Proposition~\ref{prop:triple} proves that their colourings are ordinary, virtual and ordinary,
respectively. The children of the middle region, together with the first and third regions,
are exactly the children of Definition~\ref{def:four-children}. The middle region itself
is not counted as a child. Figure~\ref{fig:four-S3} shows this assembly at level $3$.
Section~\ref{sec:four-cantor} explains how the assembly gives the stages of Cantor's
middle-thirds construction after the end margins are omitted.

\begin{figure}[ht]
\centering
\setlength{\unitlength}{\linewidth}
\begin{picture}(1,0.44)
\put(0,0){\includegraphics[width=\linewidth]{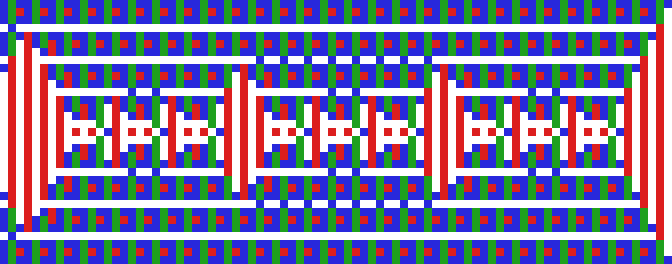}}
\put(0.202,0.415){\makebox(0,0){\small $A_1$, a box of level $2$}}
\put(0.5,0.415){\makebox(0,0){\small $A_2$, a virtual box of level $2$}}
\put(0.798,0.415){\makebox(0,0){\small $A_3$, a box of level $2$}}
\end{picture}
\caption{The main box $S_3$ of $\tau$, opening at $s=209$, in rotated coordinates. The column
index $p$ increases from left to right and the position $d$ from top to bottom. Each column
interleaves the two successive anti-diagonals $s=2p$ and $s=2p+1$: even positions $d$ take
their cells from the first, odd positions from the second. Reading a column downward thus
visits the two anti-diagonals alternately, following a staircase in the original $(x,y)$ grid.
Starting at an even position, the steps alternately increase $y$ by $1$ and decrease $x$ by $1$.
White is $0$, red $1$, blue $2$ and green $3$. Inside the
white border of $S_3$ lie two boxes of level $2$ and, between them, a virtual box of level $2$
whose walls are dashed. Each box of level $2$ contains in the same way two boxes of level $1$ and
the virtual box of level $1$, which is the middle block $\mathcal{M}$. These identifications follow
from Proposition~\ref{prop:triple}.}
\label{fig:four-S3}
\end{figure}

\begin{defn}[the family $\mathcal C$]\label{def:four-family}
For $j\ge0$ put $E_j=55+\tfrac{405}{13}(27^j-1)$, so that $E_0=55$, $E_1=865$, $E_2=22735$ and
$E_3=613225$. Write $V_j=R_{3j+3}(E_j-4)$.
The \emph{main box regions} begin with two regions of level $1$ opening at $15$ and $29$.
For each $j\ge0$, append the following regions in time order:
\begin{romenum}
\item the children of $V_j$, whose levels are $3j+2,\dots,1,0,0,1,\dots,3j+2$,
\item a region of level $3j+3$ opening at the closing time of $V_j$,
\item a region of level $3j+4$ opening $4$ rows after the region of level $3j+3$ closes.
\end{romenum}
The first region in this group has level $3j+2$ and opens at $E_j$. We call this region
the \emph{big box} of \emph{block} $j$. The main boxes of block $j$ are the regions
in the group, ending with the box of level $3j+4$.
For the purpose of the description, block $j+1$ begins one row before its big box opens.
The family $\mathcal C$ consists of the main box regions and of their descendants under
Definition~\ref{def:four-children}. Proposition~\ref{prop:diagonal} will identify the colouring
on $V_j$ as virtual, and Theorem~\ref{thm:four} will identify the regions of $\mathcal C$
as the ordinary boxes of the sequence.
\end{defn}

The box of level $3j+4$ closes $12$ rows before $E_{j+1}$, because
$E_{j+1}-E_j=2\ell_{3j+3}+\ell_{3j+4}+12=810\cdot27^j$, and the box of level $1$ opening at $29$
closes $12$ rows before $E_0$. All $E_j$ are congruent to $1$ modulo $6$. Figure~\ref{fig:four-block}
shows the beginning of block $0$.

\begin{figure}[ht]
\centering
\setlength{\unitlength}{\linewidth}
\begin{picture}(1,0.335)
\put(0,0){\includegraphics[width=\linewidth]{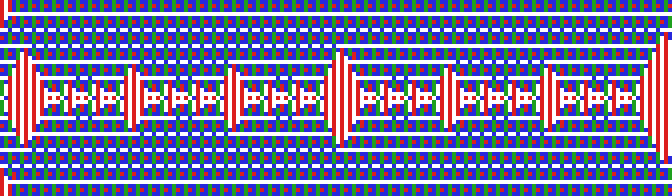}}
\put(0.0,0.31){\makebox(0,0)[l]{\small bar}}
\put(0.268,0.31){\makebox(0,0){\small $V_0$, with the children $S_2\,S_1\,S_0\,S_0\,S_1\,S_2$}}
\put(0.738,0.31){\makebox(0,0){\small $S_3$}}
\put(1.0,0.31){\makebox(0,0)[r]{\small $S_4$}}
\end{picture}
\caption{The beginning of block $0$ of $\tau$ in rotated coordinates, for $40\le s\le375$ and
$|d|\le24$. As in Figure~\ref{fig:four-S3}, each column interleaves two successive anti-diagonals,
visited alternately along a staircase in the original grid as the column is read downward.
Reading from the left, a white bar of half-width $9$ lies on row $47$.
The virtual box $V_0$ of level $3$ opens at $51$ and contains the main boxes of levels
$2,1,0,0,1,2$. The main box $S_3$ opens on the closing edge of $V_0$ at $209$,
and the main box $S_4$ opens at $371$. Proposition~\ref{prop:diagonal} identifies $V_0$ as a
virtual box. Outside $V_0$ and the main boxes lies the lattice $\Lam$.}
\label{fig:four-block}
\end{figure}

\subsection{The patterns}\label{sec:four-patterns}

\subsubsection*{The lattice $\Lam$}
$\Lam\cl sd$ depends only on $s\bmod6$ and $|d|\bmod12$ and is given by Table~\ref{tab:lam}.
The columns $|d|\equiv1,9\pmod{12}$ are entirely white, and the column $|d|\equiv5\pmod{12}$
takes the values $0,2,0$ on the rows $s\equiv1,3,5\pmod6$.

\begin{table}[ht]
\centering
\begin{tabular}{c|cccccccccccc}
$s\bmod 6$ $\backslash$ $|d|\bmod12$ & 0&1&2&3&4&5&6&7&8&9&10&11\\ \hline
0 & 3&-&2&-&2&-&2&-&2&-&3&-\\
1 & -&0&-&2&-&0&-&2&-&0&-&3\\
2 & 2&-&3&-&3&-&3&-&3&-&2&-\\
3 & -&0&-&3&-&2&-&3&-&0&-&1\\
4 & 2&-&2&-&2&-&2&-&2&-&2&-\\
5 & -&0&-&1&-&0&-&1&-&0&-&2
\end{tabular}
\caption{The value $\Lam\cl sd$ as a function of $s\bmod6$ (rows) and $|d|\bmod12$
(columns). A dash marks a combination of residues of different parity, which no cell has.}
\label{tab:lam}
\end{table}

\subsubsection*{The corridor pattern $I$}
For an integer $r$ and a column index $a\ge1$ let $I(r,a)$ be defined by the residue of $r$
modulo $6$ and the residue of $a$ modulo $4$ as follows. If $a\equiv1\pmod4$, which we call a
\emph{wall column}, $I(r,a)$ is $2,0,0$ for $r\equiv0,2,4$. If $a\equiv3\pmod4$, the
\emph{middle column} of a corridor, $I(r,a)$ is $3,1,2$ for $r\equiv0,2,4$. If $a\equiv0,2\pmod4$,
the \emph{outer columns} of a corridor, $I(r,a)$ is $2,2,3$ for $r\equiv1,3,5$. The three columns
$a=4k+2,4k+3,4k+4$ form the \emph{corridor} $k$. The wall columns are the half-widths $v_n$ of the
boxes. Inside a box, the pattern $I$, with $r$ counted from the opening edge of the box, fills
the corridors between the walls, and on a wall column it draws the dashed wall of a virtual box.

Comparing Table~\ref{tab:lam} with $I$, $\Lam\cl sd=I(s+3,|d|)$ except on the white columns
$|d|\equiv1,9\pmod{12}$ and except in the corridors $k\equiv2\pmod3$, where
$\Lam\cl sd=I(s+5,|d|)$. The region $V_j$ opens at $E_j-4\equiv3\pmod6$, so its corridor
pattern has the phase $I(s+3,\cdot)$, and outside its wall column $|d|=12j+13$ the lattice
$\Lam$ continues this pattern on the seven columns $12j+14\le|d|\le12j+20$.

\subsubsection*{Finite tables}
The description uses finite tables read from the computed sequence and printed in
Appendix~\ref{app:tables}. They specify five parts of the description:
\begin{itemize}
\item The pattern $P_0$ gives the $5$ cells of a box of level $0$, with a white border
and a red centre.
\item The pattern $P_1$ gives the $83$ cells of a box of level $1$.
\item In a box of level at least $2$, let $c_{\mathrm{mid}}$ be the closing offset of the first child
of level $1$, measured from the opening of the box. The \emph{middle block} $\mathcal{M}$ gives the $83$ cells on the rows
$c_{\mathrm{mid}},\dots,c_{\mathrm{mid}}+14$ and the columns $|d|\le5$. This block contains the two
children of level $0$. Lemma~\ref{lem:four-virtual} identifies $\mathcal{M}$ with the virtual
pattern $Q_1$.
\item The strip table gives $\tau(x,y)$ for $x\le13$. For $y\ge y_0(x)$, these
values are periodic in $y$ with period $12$, where $y_0(x)\le18$.
\item The corner table gives the values for $x,y<20$. The initial-segment table
supplies the remaining values needed for the anti-diagonals $s<56$.
\end{itemize}

\subsection{The description}\label{sec:four-desc}

\begin{defn}[the ordinary pattern $P_n$]\label{def:four-P}
The patterns $P_0$ and $P_1$ are the finite tables in Appendix~\ref{app:tables}.
For $n\ge2$, $0\le r\le\ell_n$, $|d|\le v_n$ and $v_n+r-d$ even, $P_n(r,d)$ is given by the
first applicable rule, with child offsets from Definition~\ref{def:four-children}.
\begin{romenum}
\item Border: $0$.
\item Inside a child of level $k$ opening at offset $o$: $P_k(r-o,d)$.
\item Middle block: $\mathcal{M}(r-c_{\mathrm{mid}},d)$ if $c_{\mathrm{mid}}\le r\le c_{\mathrm{mid}}+14$ and $|d|\le5$.
\item After the opening offset $o$ of a child of half-width $a$ in the first half:
$P_n(o+1,\pm(a+1))=2$ and $P_n(o+3,\pm(a+1))=1$.
\item Transition rows: for every edge of the box or of a child, at time $e'$ and
with half-width $a$, use the following values. Each value applies at cells within
the stated half-width. When two edges prescribe a cell, the rule of the wider region applies.
\[
\begin{array}{c|cccccc}
\text{row} & e'-3 & e'-2 & e'-1 & e'+1 & e'+2 & e'+3\\
\hline
\text{value} & 1 & 1 & 0 & 1 & 1 & 0\\
\text{half-width} & a-5 & a-4 & a-3 & a-1 & a-2 & a-3
\end{array}
\]
\item Otherwise $I(r,|d|)$.
\end{romenum}
\end{defn}

\begin{defn}[ordinary and virtual boxes]\label{def:four-virtual}
On the same relative cells as $P_n$, define the \emph{virtual pattern} $Q_n$ as follows.
The pattern $Q_0$ is $0$ on all five cells of $R_0(e)$. For $n\ge1$ put
\[
Q_n(r,d)=
\begin{cases}
2,& |d|=v_n,\quad 0<r<\ell_n,\quad r\equiv0\pmod6,\\
0,& |d|=v_n,\quad 0<r<\ell_n,\quad r\equiv2,4\pmod6,\\
P_n(r,d),&\text{otherwise.}
\end{cases}
\]
For a colouring $U$ of the cells, a region $R_n(e)$ is an \emph{ordinary box} if
$U\cl sd=P_n(s-e,d)$ throughout the region, and a \emph{virtual box} if
$U\cl sd=Q_n(s-e,d)$ throughout the region. Its child regions are always those of
Definition~\ref{def:four-children}. When referring to an ordinary box without specifying
its opening time, we also write $S_n$ for its pattern.
\end{defn}

\begin{lem}[properties of virtual boxes]\label{lem:four-virtual}
For $n\ge1$, $Q_n$ and $P_n$ have the same interior and the same white opening and closing
edges. Their only differences are on the walls, where $Q_n$ is blue at the offsets
$6,12,\dots,\ell_n-2$ and white at the other wall cells. Each child region of either
pattern carries the ordinary pattern of its level. Moreover, $\mathcal{M}=Q_1$.

At level $0$ the patterns differ at their centre: it is red in $P_0$ and white in $Q_0$.
No virtual box is an ordinal under the definition in Section~\ref{sec:four}.
\end{lem}
\begin{proof}
The interior and edge statements follow from Definition~\ref{def:four-virtual}. The
walls have cells only at even offsets, and the nonzero wall entries occur at the stated
multiples of $6$. Child regions lie strictly inside the parent region, so $Q_n$ and $P_n$
agree on them. Their ordinary patterns follow from rule (ii) for $n\ge2$ and from the
table of $P_1$ for $n=1$. The tables of $\mathcal{M}$ and $P_1$ differ exactly at $(6,\pm5)$ and
$(12,\pm5)$, as prescribed by $Q_1$.

For $n\ge1$ the wall cell at offset $6$ is blue, so the border is not entirely white.
For $n=0$ the only interior row is entirely white, which the definition of an ordinal
excludes.
\end{proof}

Thus the description of a virtual box as an ordinary box with dashed walls applies at
positive levels. The white centre of $Q_0$ is the separate base case at the end of the
recursion.

\begin{defn}[the description $\Delta$]\label{def:four-D}
For a cell $c=\cl sd$ with $x=(s-d)/2$ and $y=(s+d)/2$, the value $\Delta(c)$ is given by the
first applicable rule.
\begin{Fenum}{0}
\item Corner, strips and initial segment: if $x,y<20$, or $x\le13$ and $y\ge y_0(x)$, or
$y\le13$ and $x\ge y_0(y)$, or $s<56$, the value of the corresponding table.
\item Inside a main box of level $n$ opening at $e$: $P_n(s-e,d)$.
\end{Fenum}
Otherwise let $j$ be determined by $E_j-1\le s<E_{j+1}-1$.
\begin{Fenum}{2}
\item Middle block of a main pair of boxes of level $0$: $\mathcal{M}$ as in Definition~\ref{def:four-P},
relative to the closing time of the main box of level $1$ that precedes the pair.
\item Junction: for any $k\ge1$, put $e=E_k$ and $a=v_{3k+2}$. The rows
$e-10,\dots,e-1$ have the following values at cells within the stated half-widths.
\[
\begin{array}{c|cccccccccc}
s-e & -10 & -9 & -8 & -7 & -6 & -5 & -4 & -3 & -2 & -1\\
\hline
\text{value} & 1 & 0 & 0 & 1 & 1 & 0 & 0 & 1 & 1 & 0\\
\text{half-width} & a-4 & a-3 & a & a-1 & a & a+1 & a+4 & a+3 & a+2 & a+1
\end{array}
\]
\item Transition rows around the edges of the main boxes, as in Definition~\ref{def:four-P}(v).
\item After the opening $e$ of a main box with $e\equiv1\pmod6$ and half-width $a$:
$\Delta\cl{e+1}{\pm(a+1)}=2$ and $\Delta\cl{e+3}{\pm(a+1)}=1$.
\item Corridors and dashed walls of the block: if $1\le|d|\le12j+13$, then $\Delta(c)=I(s+3,|d|)$.
\item Otherwise $\Delta(c)=\Lam(c)$.
\end{Fenum}
\end{defn}

In words, in block $j$ the rules (F2) to (F7) draw the virtual box $V_j$ with its children, the
main boxes of levels $3j+3$ and $3j+4$, and the junction before the next block, and rule (F8)
draws the lattice $\Lam$ elsewhere. Proposition~\ref{prop:diagonal} states this precisely.

\subsection{The theorem}\label{sec:four-thm}

\begin{thm}\label{thm:four}
Let $\Delta$ be the function on cells of Definition~\ref{def:four-D} and $\mathcal C$ the family
of Definition~\ref{def:four-family}.
\begin{alphenum}
\item $\tau(x,y)=\Delta\cl{x+y}{y-x}$ for all $x,y\ge0$.
\item The ordinals of $\tau$ are the boxes of $\mathcal C$. All boxes of level $n$ are
translates of each other, and the ordinals directly inside a box of level $n+1$ are its
children, which are translates of $S_n,\dots,S_0,S_0,\dots,S_n$ in this order.
\item The ordinals contained in no other ordinal are the main boxes. Along the diagonal,
in increasing time $s=x+y$, the first two main boxes have level $1$ and open at times
$15$ and $29$. The remaining main boxes occur in successive groups indexed by
$j=0,1,2,\dots$. In group $j$, the levels occur in the order
\[
\underbrace{3j+2,3j+1,\dots,1,0,0,1,\dots,3j+2}_{6j+6\text{ children of }V_j},
\quad 3j+3,\quad 3j+4.
\]
The virtual box $V_j$ has level $3j+3$ and opens at $E_j-4$, where
$E_j=55+\tfrac{405}{13}(27^j-1)$.
The main box of level $3j+3$ opens on the closing edge of $V_j$.
The main box of level $3j+4$ opens $4$ rows after the main box of level $3j+3$ closes.
The next virtual box $V_{j+1}$ opens $8$ rows after the main box of level $3j+4$ closes.
Each level occurs infinitely often among the main boxes.
\end{alphenum}
\end{thm}

The next two propositions follow from the definitions and describe the structure behind the
theorem. They are also the main tools of its proof.

\begin{prop}[assembly of the boxes]\label{prop:triple}
Let $n\ge2$, $\ell=\ell_{n-1}$, $w=v_{n-1}$ and $v=v_n$, and let $0\le r\le\ell_n$ and $|d|\le v$.
\begin{alphenum}
\item If $4\le r\le4+\ell$ and $|d|\le w$, then $P_n(r,d)=P_{n-1}(r-4,d)$. If
$4+2\ell\le r\le4+3\ell$ and $|d|\le w$, then $P_n(r,d)=P_{n-1}(r-4-2\ell,d)$.
\item If $4+\ell\le r\le4+2\ell$ and $|d|\le w$, then
$P_n(r,d)=Q_{n-1}(r-4-\ell,d)$. In particular, the middle region is a virtual box in the
sense of Definition~\ref{def:four-virtual}. Off its walls the value is
$P_{n-1}(r-4-\ell,d)$. On the walls of the middle region, strictly between its edges,
the value is $I(r,w)$.
\item At the remaining cells, $P_n(r,d)=0$ if $r\in\{0,\ell_n\}$ or $|d|=v$. The rows $r=1,2,3$
are red on $|d|\le v-1$, red on $|d|\le v-2$ and white on $|d|\le v-3$, respectively.
Before the closing edge, the rows
$r=\ell_n-1,\ell_n-2,\ell_n-3$ are white on $|d|\le v-3$, red on $|d|\le v-4$ and red on
$|d|\le v-5$, respectively. The two exceptional values are
\[
P_n(5,\pm(w+1))=2,\qquad P_n(7,\pm(w+1))=1.
\]
At all other remaining cells, $P_n(r,d)=I(r,|d|)$.
\end{alphenum}
The decomposition into three patterns also holds for $n=1$: the regions at offsets $4$, $6$ and $8$
carry $P_0$, $Q_0$ and $P_0$, respectively. The same three constituent patterns occur
inside $Q_n$ for every $n\ge1$.
\end{prop}

\begin{proof}
By the definition of the children and $\ell_n=3\ell+8$, the children of $P_n$ are, in time
order, a box of level $n-1$ at the offset $4$, the children of $P_{n-1}$ shifted by $4+\ell$, and
a box of level $n-1$ at the offset $4+2\ell$. Part (a) is rule (ii).

For (b) we compare the rules
of Definition~\ref{def:four-P} for $P_n$ at $(r,d)$ and for $P_{n-1}$ at $(r-4-\ell,d)$. Rule
(ii) concerns the same children. For $n\ge3$, the middle block of $P_n$ is the
middle block of $P_{n-1}$ shifted by $4+\ell$. For $n=2$, the middle block of $P_2$
occupies the rows $18,\dots,32$ and the columns $|d|\le5$, which form the region of (b).
Rule (iv) concerns the same first-half children. The edges seen by
rule (v) are those of the same children and two edges of half-width $w$ at the offsets $4+\ell$
and $4+2\ell$, which for $P_{n-1}$ are its own edges. Rule (vi) gives the same value because
$4+\ell\equiv0\pmod6$.

For $n\ge3$, consider the walls $|d|=w$ strictly between these two edges.
No child reaches these columns, and the middle block lies in $|d|\le5<w$.
Rule (iv) concerns even columns, while the transition rows of (v) reach only
$|d|\le w-1$. Thus rule (vi) gives $I(r,w)$ on the walls of the middle region,
where $P_{n-1}$ has its ordinary wall. Since $4+\ell\equiv0\pmod6$, these wall values are exactly those of
$Q_{n-1}(r-4-\ell,d)$, and the shared edges are white. For $n=2$ the middle region carries
$\mathcal{M}=Q_1$ by Lemma~\ref{lem:four-virtual}.

Part (c) follows in the same way from the rules
(i), (iv), (v) and (vi), since the
transition rows of the edges of the children, of half-width at most $w$, do not reach the
columns $|d|>w$.

For $n=1$ the restrictions of $P_1$ to the rows $4,\dots,6$ and $8,\dots,10$, with
$|d|\le1$, are $P_0$, while all five cells of the region on rows $6,\dots,8$ are white,
giving $Q_0$. Finally, all three constituent regions lie strictly inside $R_n(e)$,
where $P_n=Q_n$ for $n\ge1$, so the assertion for $Q_n$ follows as well.
\end{proof}

\begin{prop}[the diagonal]\label{prop:diagonal}
Let $j\ge0$ and let $V_j=R_{3j+3}(E_j-4)$ be the region of
Definition~\ref{def:four-family}.
\begin{alphenum}
\item On the closed region of $V_j$, $\Delta\cl sd=Q_{3j+3}(s-E_j+4,d)$, so $V_j$ is a
virtual box for $\Delta$. Equivalently, $\Delta\cl sd=P_{3j+3}(s-E_j+4,d)$ except on the walls
$|d|=12j+13$ strictly between the two edges of $V_j$, where $\Delta\cl sd=I(s+3,12j+13)$. This is
$2$ for $s\equiv3$ and $0$ for $s\equiv1,5\pmod6$, so the walls of $V_j$ are dashed.
\item The main box of level $3j+4$ closes at $E_{j+1}-12$, and the big box of
block $j+1$ opens at $E_{j+1}$. The eleven intervening rows contain a white bar of
half-width $v_{3j+5}$ at time $E_{j+1}-8$ and the opening edge of $V_{j+1}$,
of half-width $v_{3j+6}$, at time $E_{j+1}-4$.
The other rows in this interval are transition rows of the edges at times
$E_{j+1}-12$, $E_{j+1}-8$, $E_{j+1}-4$ and $E_{j+1}$. Where transition rules overlap,
the rule for the edge with greatest half-width applies.
\item The cells of block $j$ that lie neither in $V_j$, nor in a main box, nor on the rows of (b)
are given by the transition rows around the edges of the main boxes of levels $3j+3$ and
$3j+4$, and elsewhere by $\Lam$.
\end{alphenum}
\end{prop}

\begin{proof}
(a) The children of $V_j$ are the main boxes of levels $3j+2,\dots,0,0,\dots,3j+2$ by
Definition~\ref{def:four-family}, and on their closed regions rule (F2) and rule (ii) give the same
patterns. For the other cells of $V_j$ we compare the rules of Definition~\ref{def:four-D} with
those of Definition~\ref{def:four-P} for $P_{3j+3}$ at $(s-E_j+4,d)$. Rule (F3) is rule (iii),
since the main box of level $1$ before the main pair of level $0$ is the first-half child of level
$1$ of $V_j$. Rule (F6) is rule (iv), since the main boxes opening at times $\equiv1\pmod6$ inside
$V_j$ are its first-half children. The rows $E_j-4,\dots,E_j-1$ of rule (F4) are the opening edge
of $V_j$, white on $|d|\le12j+13$, and the rows red on $|d|\le12j+12$, red on $|d|\le12j+11$ and
white on $|d|\le12j+10$, which are the transition rows of that edge in rule (v). Rule (F5) gives the
transition rows of the children and of the closing edge of $V_j$, which is the opening edge of the
main box of level $3j+3$. Rule (F7) is rule (vi), since $s-E_j+4\equiv s+3\pmod6$, and at $d=0$
the cells of $V_j$ are covered by the earlier rules, so that (F8) does not occur in $V_j$. On the
walls strictly between the edges, rule (F7) gives $I(s+3,12j+13)$, where $P_{3j+3}$ has its wall.
Part (b) is rule (F4) with the half-width $a=v_{3j+5}$ of the big box of block $j+1$, and (c) is
rules (F5) and (F8).
\end{proof}

Theorem~\ref{thm:four}(a) transfers these identifications from $\Delta$ to the recurrent
sequence $\tau$. Thus the virtual boxes inside ordinary boxes are identified by
Proposition~\ref{prop:triple}, and the virtual boxes $V_j$ along the diagonal by
Proposition~\ref{prop:diagonal}.

\begin{cor}[the mirrored ordinals with one initial symbol]\label{cor:four-ordinals}
For each main box $B$ of level $n$, the ordered rooted tree of ordinary boxes contained
in $B$, including $B$ itself, is isomorphic to the occurrence tree of $s_n$. Each child
corresponds to a particular occurrence in the defining list, with time determining the
order of children. Ignoring the order and the
repetition, the children of a box of level $n$ have exactly the levels $0,\dots,n-1$, so that
\emph{$K$ is a child of $B$} becomes $k\in n$ for the von Neumann ordinals $n=\{0,\dots,n-1\}$.
\end{cor}

\begin{proof}
By Theorem~\ref{thm:four}(b) and the description of the children in
Section~\ref{sec:four-boxes}.
\end{proof}

\begin{cor}[number of ordinary boxes]\label{cor:four-count}
An ordinary box of level $n\ge0$ contains exactly $3^n$ ordinary boxes, including itself
and all its descendants. A virtual box of level $n$ contains exactly $3^n-1$ ordinary boxes.
\end{cor}
\begin{proof}
Let $N_n$ count the ordinary boxes in an ordinary box of level $n$, including the outer box.
Theorem~\ref{thm:four}(b) gives $N_0=1$ and
$N_n=1+2\sum_{k=0}^{n-1}N_k$ for $n\ge1$. Subtraction gives $N_n=3N_{n-1}$, hence $N_n=3^n$.
By Lemma~\ref{lem:four-virtual}, a virtual box has the same ordinary children and their
descendants. Any ordinary box of smaller level has half-width at most $v_n-4$, where
$Q_n=P_n$, so no further ordinary box appears when the outer walls are changed.
The virtual box itself is not an ordinal, so its count is $N_n-1=3^n-1$.
\end{proof}

\subsection{Relation with Cantor's middle-thirds construction}\label{sec:four-cantor}

The assembly of an ordinary box of level $n$ gives the first $n$ steps of Cantor's
middle-thirds construction when we select its constituent regions as follows. Project
the regions onto the time axis. At each level above $0$, omit the four-row margins at
both ends and retain the two outer ordinary boxes. Remove the open middle span and
discard all boxes inside its virtual region. Repeating this operation inside each retained box
gives the same ordered arrangement of retained intervals and removed middle intervals
as the Cantor construction. The interval lengths agree after the recursive normalization
described below.

For comparison, let $K_0=[0,1]$ and obtain $K_{k+1}$ by deleting the open middle third
of each interval in $K_k$. Equivalently,
\[
K_{k+1}=\tfrac13 K_k\,\cup\,\bigl(\tfrac23+\tfrac13 K_k\bigr).
\]
Thus $K_k$ consists of $2^k$ closed intervals of length $3^{-k}$.
The intersection $\bigcap_{k\ge0}K_k$ is the Cantor set, whose description by ternary expansions with
digits $0$ and $2$ appears in Cantor's 1883 paper \cite[p.~590]{Cantor83}.
Hutchinson gives the formulation by two contractions and its finite approximations
in \cite[Section~3.3(1)]{Hut81}.

To check the correspondence, consider a box of level $m\ge1$ opening at time $e$, and
put $L=\ell_{m-1}$. Proposition~\ref{prop:triple} gives three consecutive constituent
regions with the following patterns and time spans:
\[
\begin{array}{c|c}
\text{pattern}&\text{time span}\\ \hline
P_{m-1}&[e+4,e+4+L]\\
Q_{m-1}&[e+4+L,e+4+2L]\\
P_{m-1}&[e+4+2L,e+4+3L].
\end{array}
\]
Each span has length $L$, and adjacent spans share an endpoint. The identity
$\ell_m=4+3L+4$ accounts for the margins before and after these regions.
Assign the whole box to $[0,1]$ and its three constituent regions to the left, middle
and right thirds. Repeat this assignment inside each retained ordinary box, using its
assigned interval in place of $[0,1]$. This recursively omits the end margins and
normalizes each triple. A single affine change of the original time coordinate would
leave the ratio $L/\ell_m=L/(3L+8)$ unchanged.

After $k$ selections, where $0\le k\le n$, there are $2^k$ retained ordinary boxes of
level $n-k$. Their assigned intervals form $K_k$. Indeed, each selection replaces a
retained box by its two outer constituents and removes the open middle interval
assigned to its virtual constituent. At step $k\ge1$ this excludes $2^{k-1}$ virtual
regions of level $n-k$. After $n$ steps, the $2^n$ retained level-$0$ boxes therefore
correspond to the component intervals of $K_n$, and the $2^n-1$ excluded virtual
regions correspond to the intervals removed during its construction.

\begin{samepage}
For example, a box of level $2$ first splits into $P_1,Q_1,P_1$. Discarding the middle
$Q_1$ and splitting each retained $P_1$ into $P_0,Q_0,P_0$ leaves four ordinary boxes
of level $0$. Their assigned intervals are
\[
K_2=[0,\tfrac19]\,\cup\,[\tfrac29,\tfrac13]\,\cup\,
[\tfrac23,\tfrac79]\,\cup\,[\tfrac89,1].
\]
The original coloured diagram also contains ordinary boxes inside virtual regions,
as Lemma~\ref{lem:four-virtual} shows. In this example the discarded $Q_1$ contains two
further ordinary boxes of level $0$, both excluded at the first selection.
\end{samepage}

\subsection{Proof of Theorem~\ref{thm:four}(a)}\label{sec:four-proof}

The axis entries in the corner and strip tables give $\Delta\cl s{\pm s}=1$.
By Lemma~\ref{lem:unique}, applied to the alphabet $A$ and the rule $g$, it therefore suffices to show that
$\Delta$ is correct at each cell $c=\cl sd$ with $|d|\le s-2$, that is
$\Delta(c)=g(\Delta(\cx),\Delta(\cy),\Delta(\cz))$. We call the quadruple
$\omega(c)=(\Delta(c),\Delta(\cx),\Delta(\cy),\Delta(\cz))$ the \emph{window} of $c$. Whether
$\Delta$ is correct at $c$ depends only on $\omega(c)$. It therefore suffices to show that each
window that occurs somewhere occurs in a bounded region, and to verify $\Delta$ there by
computer.

\begin{lem}[windows]\label{lem:windows}
Let $\mathcal{R}$ be the set of cells $\cl sd$ with $s<613\,224$ that lie in the diagonal band
$|d|\le60$ or in the axis margins $x\le15$, $y\le15$. Each window centred at a cell with $|d|\le s-2$ occurs at such a cell of $\mathcal{R}$.
\end{lem}

Block $3$ begins at $E_3-1=613\,224$, so $\mathcal{R}$ consists of the cells of the initial segment and of
blocks $0$ to $2$ that lie in the band or in the margins.

\begin{proof}
Let $c=\cl sd$, with $x=(s-d)/2$ and $y=(s+d)/2$.

\emph{Cells near the axes.} Let $x\le15$. If $y\le55$ then $s\le70$ and $|d|\le55$, so $c$ lies
in $\mathcal{R}$. Let $y\ge56$. The cells with $x\le13$ lie in the strip, where $\Delta$ depends only on
$(x,\,y\bmod12)$. A cell with $x=14$ or $x=15$ lies at distance $|d|=y-x\ge41$ from the diagonal,
while the rules (F2) to (F7) concern the cells of block $j$ with $|d|\le12j+25$. For $j=0$ this is
less than $41$, and for $j\ge1$ a cell of block $j$ has $s\ge E_j-1\ge30\cdot27^j$, hence
$y-x\ge30\cdot27^j-30>12j+25$. These cells are therefore given by (F8), and their values depend
only on $s\bmod6$ and $|d|\bmod12$, which are functions of $x$ and $y\bmod12$. The windows of $c$
and of $(x,\,y-12k)$ therefore coincide for the $k$ with $56\le y-12k<68$, which is a cell of $\mathcal{R}$.
Cells with $y\le15$ are symmetric. From now on let $x,y\ge16$ and $s\ge58$, since the cells with
$s<58$ lie in $\mathcal{R}$. Then $c$ and its parents lie outside the corner, the strips and the initial
segment.

\emph{Far exterior.} If $c$ and its parents are given by rule (F8), then $\omega(c)$ depends only
on $s\bmod6$ and $|d|\bmod12$, together with the sign of $d$. Here $d\ne0$,
since diagonal cells lie in the boxes or their junction transitions.
Choose the representative with
$s'=400+((s-400)\bmod6)$ and $d'=30+((|d|-30)\bmod12)$. It lies in block $0$,
and all four cells are given by (F8). For negative $d$, reflect the representative
to preserve the order of the two side parents.

\emph{A map for the time coordinate.} We first specify the shortening used below.
For lengths $L,L'\ge21$ with $L\equiv L'\pmod6$, define, for $0\le u<L$,
\[
\phi_{L,L'}(u)=
\begin{cases}
u,&0\le u\le9,\\
L'-(L-u),&1\le L-u\le5,\\
10+((u-10)\bmod6),&10\le u\le L-6.
\end{cases}
\]
The endpoint cases preserve the distance to the corresponding endpoint.
The middle part of the interval is represented by one of the six rows $10,\dots,15$.
In all cases $\phi_{L,L'}(u)\equiv u\pmod6$. We map the centre of a window and then
use its three parents. We do not apply $\phi$ separately to the four cells.

\emph{Cells in a box.} Let $W_m$ be the set of ordered windows obtained from $P_m$
when all four cells lie in its closed region. Thus their centres satisfy
$2\le r\le\ell_m$, $|d|+1\le v_m$ and $v_m+r\equiv d\pmod2$.
The tables give $W_0\subseteq W_1$, and the assembly in
Proposition~\ref{prop:triple} gives $W_{m-1}\subseteq W_m$ for $m\ge2$.
At the cells currently under consideration all four times are at least $56$, so
rule (F2) identifies a window contained in a main box with a window of $P_m$.
The boxes in the initial segment have already been covered by $\mathcal{R}$.

Put $L=\ell_{m-1}$ and $w=v_{m-1}$. The three constituent regions of $P_m$ have
intervals $[4+iL,4+(i+1)L]$, for $i=0,1,2$, and half-width $w$.
If a window is wholly in the first or third region, or wholly in the middle region
off its walls, it belongs to $W_{m-1}$. Call the other windows \emph{frame windows}.
A frame window either crosses an endpoint of these intervals, lies in one of the
outer margins, or has $|d|\ge w-1$.
In the first two cases its cells inside a constituent region are on its edge or
within one row of that edge. In the last case all four positions have absolute value
at least $w-2$, whereas the children of that constituent have half-width at most $w-4$.
Consequently Proposition~\ref{prop:triple}(c), applied to the box and its three
constituents, determines every cell of a frame window without further recursion.

Here is the complete information required by those frame rules. For a region of
half-width $v$, an ordinary wall has value $0$, a virtual wall has value $I$,
and either kind of region has value $0$ on its opening and closing edges.
The rows beside an edge at time $a$ use the following cutoffs:
\[
\begin{array}{c|rrrrrr}
s-a&-3&-2&-1&1&2&3\\ \hline
\text{value}&1&1&0&1&1&0\\
\text{largest }|d|&v-5&v-4&v-3&v-1&v-2&v-3.
\end{array}
\]
For an internal frame, only the rows on the appropriate side of its own edge are
used. The two exceptional cells after an opening have relative coordinates
$(5,\pm(v-3))$ and $(7,\pm(v-3))$, with values $2$ and $1$.
All remaining frame values are $I(r,|d|)$, with $r$ counted from the opening of
the relevant region. This value depends on $r\bmod6$ and
$|d|\bmod4$. The kind of region, distances to its edges and walls, and these two
residues therefore determine the ordered window.

For $m\ge6$, set $L'=\ell_{m-2}$. Map a frame centre $(r,d)$ to $(r',d')$ in level
$m-1$ by
\[
r'=\begin{cases}
r,&r<4,\\
4+iL'+\phi_{L,L'}(r-4-iL),&4+iL\le r<4+(i+1)L,\quad i=0,1,2,\\
\ell_{m-1}-(\ell_m-r),&r\ge4+3L,
\end{cases}
\]
and
\[
d'=\begin{cases}
d,&|d|\le7,\\
d-4\operatorname{sgn}(d),&|d|\ge8.
\end{cases}
\]
Every corresponding half-width decreases by $4$. For $|d|\ge8$, the absolute
positions of all four cells decrease by $4$, so every wall distance and cutoff in
the table is preserved. For $|d|\le7$, all four positions have absolute value at
most $8$. Such a frame window is near horizontal edges, where all relevant rows
reach these positions in both boxes: the smallest half-width in the target is
$v_{m-2}\ge17$, and the largest subtraction in the table is $5$.

Near an endpoint, the time map preserves the offsets of all four cells, including
those on the other side of a shared edge. Away from endpoints, the four relative
times are at least $8$ after an opening and at least $6$ before a closing. No
exceptional or transition row applies there, so the residues modulo $6$ suffice.
The congruence $L\equiv L'\pmod6$ also preserves the phase when an endpoint belongs
to two constituents. Thus the displayed map preserves each of the four values
and takes a frame window to a frame window of level $m-1$.
It follows inductively that $W_m\subseteq W_6$ for every $m$.

\emph{Cells in a virtual box of the diagonal.} A window wholly in $V_j$ off its
walls belongs to $W_{3j+3}$ by Proposition~\ref{prop:diagonal}(a), hence to $W_6$.

\emph{Cells near the block structure.} We give a separate map for the remaining
windows. Set
\[
\theta_j=E_j-4,\qquad L=\ell_{3j+3},\qquad v=v_{3j+3}=12j+13.
\]
The intervals between successive virtual openings satisfy $\theta_{j+1}-\theta_j=5L+20$.
Relative to $\theta_j$, Proposition~\ref{prop:diagonal} gives the following regions and
junction edges:
\[
\begin{array}{c|c|c}
\text{region or edge}&\text{time interval or time}&\text{half-width}\\ \hline
V_j&[0,L]&v\\
\text{ordinary box of level }3j+3&[L,2L]&v\\
\text{ordinary box of level }3j+4&[2L+4,5L+12]&v+4\\
\text{next bar}&5L+16&v+8\\
\text{opening of }V_{j+1}&5L+20&v+12.
\end{array}
\]
The preceding ordinary box closes at $\theta_j-8$, and the preceding bar lies at
$\theta_j-4$, with half-widths $v-8$ and $v-4$. These also determine the parents of
windows centred at $\theta_j$ or $\theta_j+1$.

Each remaining window meets a time endpoint or one of the listed walls, or lies
in a short gap. In a region whose wall is met, its children are four columns further
in, so the frame rules just listed apply. At a crossed endpoint, only the first
or last two rows of that region are needed. In the gaps, the edge-transition table
applies, with the largest half-width taking priority where rows overlap.
Outside these regions and transitions the value is $\Lam$.
The listed regions and transitions have half-width at most $v+12$, so a window
with at least one such cell has centre $|d|\le v+13$.
Thus the additional information, beyond that used for internal frames, is the
absolute time modulo $6$, the absolute position modulo $12$, and the widths of
the neighbouring edges in the table. This covers the four-row gap and both
four-row parts of the junction containing the bar.

For $j\ge3$, put $L'=\ell_{3j}$ and write $s=\theta_j+r$, with $0\le r<5L+20$.
Choose $s'=\theta_{j-1}+r'$ by the following table:
\[
\begin{array}{c|l}
\text{range of }r&r'\\ \hline
{[0,L)}&\phi_{L,L'}(r)\\
{[L,2L)}&L'+\phi_{L,L'}(r-L)\\
{[2L,2L+4)}&2L'+(r-2L)\\
{[2L+4,5L+12)}&2L'+4+\phi_{3L+8,3L'+8}(r-2L-4)\\
{[5L+12,5L+20)}&5L'+12+(r-5L-12).
\end{array}
\]
For the position use $d'=d$ when $|d|\le13$ and
$d'=d-12\operatorname{sgn}(d)$ when $|d|\ge14$.
All corresponding widths, including the preceding junction widths, decrease by
$12$. In the latter case the absolute positions of all four cells decrease by
$12$, preserving the cutoffs and the periods of both $I$ and $\Lam$.
In the former case the window meets horizontal edges or their transitions, which
are constant across $|d|\le14$. Even the smallest preceding edge in the target
has half-width $29$, so its narrowest transition reaches $|d|=24$.

The time map preserves each short gap and the offsets from both of its endpoints.
Within a long region the same endpoint and residue argument as for internal frames
applies. Since $\theta_j\equiv \theta_{j-1}\pmod6$ and $L\equiv L'\pmod6$, absolute lattice
phases are preserved as well. These facts give equality of the four ordered values.
Repeating the map places the remaining window between $\theta_0$ and $\theta_3$, at a position
with $|d|\le12\cdot2+26=50$. The windows before $s=58$ were already in $\mathcal{R}$.

Finally, $W_6$ is realised inside the main box of level $6$ in block $1$, which lies
in $\mathcal{R}$. Together with the axis-margin and exterior cases, this proves the lemma.
\end{proof}

\begin{prop}[computer verification]\label{prop:four-machine}
\begin{alphenum}
\item $\Delta$ is correct at each cell $\cl sd$ of the region $\mathcal{R}$ with $|d|\le s-2$. This is a check of $55\,492\,541$ recurrence windows, with
$2\le s<613\,224$ and either $|d|\le60$ or $x\le15$ or $y\le15$.
\item The argument triples of $g$ that occur at the cells of part (a) are the $56$ triples of $A^3$
other than the eight listed in Table~\ref{tab:g}.
\item $\Delta$ agrees with $\tau$, computed from the recurrence, at all cells with $x,y<40000$.
\end{alphenum}
\end{prop}
\begin{proof}
By computation, items 6 and 7 of Section~\ref{sec:computer}. Part (c) is an independent check and
is not used in the proofs.
\end{proof}

\begin{proof}[Proof of Theorem~\ref{thm:four}(a)]
Let $c$ be a cell with $|d|\le s-2$. By Lemma~\ref{lem:windows} its window occurs at a cell $c'$ of
$\mathcal{R}$, and $\Delta$ is correct at $c'$ by Proposition~\ref{prop:four-machine}(a). Hence $\Delta$ is
correct at $c$, and $\Delta=\tau$ by Lemma~\ref{lem:unique}.
\end{proof}

\begin{cor}\label{cor:four-triples}
The triples $(\tau(x-1,y),\tau(x-1,y-1),\tau(x,y-1))$ with $x,y\ge1$ are exactly the $56$ triples
of $A^3$ other than the eight listed in Table~\ref{tab:g}.
\end{cor}
\begin{proof}
The triple at a cell consists of the last three entries of its window. By Lemma~\ref{lem:windows}
each window occurs in $\mathcal{R}$, and the triples that occur there are given by
Proposition~\ref{prop:four-machine}(b).
\end{proof}

\subsection{Proof of Theorem~\ref{thm:four}(b) and (c)}\label{sec:four-bc}

\begin{lem}[separation from the axes]\label{lem:four-margins}
For $s\ge58$, a centred white segment on row $s$ has half-width at most $s-32$.
Consequently every ordinal opening at $s\ge58$ lies in $x,y\ge16$.
\end{lem}
\begin{proof}
In block $j$, the rules (F2) to (F7) affect only $|d|\le12j+25$, including the
junction before block $j+1$. For $s\ge58$ in block $0$, we have $s-32>25$. For $j\ge1$,
$s\ge E_j-1\ge30\cdot27^j$ gives $s-32>12j+25$.
Thus the cells $[s,s-30]$ and $[s,s-32]$, with $x=15$ and $x=16$, are given by
$\Lam$. On even rows both are non-white. On odd rows, their pairs of values for
$s\equiv1,3,5,7,9,11\pmod{12}$ are, respectively,
\[
(2,0),\ (0,3),\ (2,0),\ (0,3),\ (3,0),\ (0,1).
\]
They are never both white. A centred white segment therefore has half-width less than
$s-30$, hence at most $s-32$ by parity. If an ordinal opens at $s_0$ with half-width
$W\le s_0-32$, each of its cells has $x,y\ge(s_0-W)/2\ge16$.
\end{proof}

By part (a) we work with $\Delta$. We use three facts on the white cells with $x,y\ge16$ and
$s\ge56$, which follow from Definition~\ref{def:four-D} and Propositions~\ref{prop:triple}
and~\ref{prop:diagonal}.
\begin{romenum}
\item On an even column $|d|=W$ the cells $\cl sW$ and $\cl{s+2}W$ are never both white.
The pattern $I$, the lattice $\Lam$ and the rules (iv) and (F6) never give $0$ on even
columns. A white cell in an even column must therefore lie on a white transition row
$e-1$ or $e+3$ of an edge $e$, on a white junction row, or at one of the positions
$(3,0)$, $(3,\pm2)$, $(7,0)$, $(11,0)$ and $(13,\pm2)$ in $P_1$ or $\mathcal{M}$.
No two of these white cells are two rows apart in the same column.
\item On an odd row, the cells with $|d|=1$ or $|d|=3$ are white only on the row of an edge of
half-width at least $|d|$, where an edge means an edge of a box of $\mathcal C$, an edge of a virtual
box $V_j$, or a bar. On odd rows the transition rows are red, the pattern $I$ is not white on the
middle columns $a\equiv3\pmod4$, the lattice $\Lam$ is not white on $|d|\equiv3\pmod{12}$, and the
column $|d|=1$ is covered by the boxes and their transition rows. Moreover, if $a$ is the
greatest half-width of an edge on that row, its adjacent cells $[s,\pm(a+2)]$ are
non-white. For edges inside ordinary or virtual regions these cells lie on middle
corridor columns. At level $0$, their non-white values also follow from the tables of
$P_1$ and $\mathcal{M}$.
At the remaining main edges, of levels $3j+3$ and $3j+4$, one has
$s\equiv1,5\pmod6$ and $a+2\equiv3,7\pmod{12}$, giving the values $2$ and $1$ in
the table of $\Lam$. At a virtual opening, $s\equiv3\pmod6$ and
$a+2\equiv3\pmod{12}$ give $3$. At a bar, $s\equiv5\pmod6$ and
$a+2\equiv11\pmod{12}$ give $2$.
\item The row of such an edge of half-width $a$ is white on $|d|\le a$, and the row two rows after
it is red on $|d|\le a-2$.
\end{romenum}

\begin{proof}[Proof of Theorem~\ref{thm:four}(b) and (c)]
Let $B^\ast$ be an ordinal with half-width $W$, top row $s_0$ and bottom row $s_1>s_0$. Then
$s_0\equiv s_1\equiv W\pmod2$, the rows $s_0$ and $s_1$ are white on $|d|\le W$, and the cells
$\cl s{\pm W}$ with $s_0\le s\le s_1$ and $s\equiv W\pmod2$ are white. The ordinals with $s_0<58$
are the boxes of $\mathcal C$ by item 8 of Section~\ref{sec:computer}. To justify its
finite cutoff, we check the $53$ possible centred white opening segments with
$s_0<58$. For each segment, a prospective wall meets a non-white cell by time $111$
(item~\ref{check:arithmetic-and-boundaries}), so no ordinal with such an opening can extend beyond
the enumeration range.
For $s_0\ge58$, Lemma~\ref{lem:four-margins} places the whole ordinal in $x,y\ge16$,
so (i) applies at its wall and makes $W$ odd. By (ii), applied at $|d|=3$ if $W\ge3$
and at $|d|=1$ otherwise, $s_0$ is an edge time. Let $E$ be a widest edge there,
of half-width $a$. Its adjacent cell at $|d|=a+2$ is non-white, so $W\le a$.

Suppose first that $s_1=s_0+2$. The closing row is white at $|d|=1$, so $a-2<1$
by (iii), and $E$ is an edge of a box of level $0$. If $E$ is an opening edge,
$B^\ast$ is that box. If $E$ closes the first of the two level-$0$ boxes in a middle
block, then $B^\ast$ would be the virtual box between the pair. Its centre is white,
so the condition on interior rows excludes this case. If $E$ closes the second
box in the pair, the cells at $|d|=1$ are green two rows later, which prevents a
white closing edge at $s_0+2$.

Now let $s_1\ge s_0+4$. The wall cell $\cl{s_0+2}W$ is white, so $W\ge a-1$
by (iii). Since $W\le a$ and both half-widths are odd, $W=a$.
If $a=1$, the interior row $s_0+2$ has no cell with $|d|<1$, contrary to the
condition on interior rows. Thus $a\ge5$.

If $E$ opens a box $B$ of $\mathcal C$, no row strictly between the edges of $B$
can close an ordinal of half-width $W$. On each such row of the appropriate parity,
the cells at $|d|=W-2$ are non-white. They are red at offset $2$ from the opening
and otherwise follow a middle column of $I$, or the non-white entries of $P_1$
when $B$ has level $1$. Hence $s_1$ is at least the closing time of $B$.
The closing row of $B$ is white on $|d|<W$, so that row cannot be an interior row
of $B^\ast$. Therefore $s_1$ is the closing time of $B$, and $B^\ast=B$.

If $E$ is not the opening edge of an ordinary box, there are four possibilities:
\begin{itemize}
\item $E$ closes a first sub-box $A_1$. The column $|d|=W$ then follows the dashed
wall of the middle virtual box $A_2$, with a blue cell at time $s_0+6$.
\item $E$ closes a box and is followed four rows later by an edge of half-width
$W+4$. The transition row before that edge makes $\cl{s_0+2}W$ red.
\item $E$ is a bar. Rule (F4) makes $\cl{s_0+2}W$ red.
\item $E$ opens $V_j$. Its dashed wall has a blue cell at time $s_0+6$.
\end{itemize}
A red wall cell at $s_0+2$ prevents $B^\ast$ from extending to $s_1\ge s_0+4$.
In either case with a dashed wall, the blue cell at $s_0+6$ forces $s_1=s_0+4$.
At time $s_0+4$, the first child of the virtual box opens with half-width $W-4$.
The cells at $|d|=W-2$ lie on a middle corridor column and are non-white, so this
row cannot be the closing edge of $B^\ast$.

Hence the ordinals are the boxes of $\mathcal C$, and the ordinals directly inside a box are its
children. All boxes of level $n$ are translates of each other by the rules (F2) and (ii), and the
children are described in Section~\ref{sec:four-boxes}. This proves (b). The main boxes are the
boxes of $\mathcal C$ that are not children of other boxes, and their positions are those of
Definition~\ref{def:four-family}, which proves (c).
\end{proof}

\section{Non-automaticity}\label{sec:nonauto}

For an integer $k\ge2$, the \emph{$k$-kernel} of a sequence $u:\N\to A$ is the set
of sequences $n\mapsto u(k^en+b)$ with $e\ge0$ and $0\le b<k^e$.
The sequence is $k$-automatic if and only if this set is finite
\cite[Theorem~6.6.2]{AS03}.
For a double sequence $T:\N^2\to A$, the $k$-kernel consists of the double sequences
$(x,y)\mapsto T(k^ex+a,k^ey+b)$ with $e\ge0$ and $0\le a,b<k^e$.
Again, finiteness of the kernel characterises $k$-automaticity
\cite[Chapter~14]{AS03}. We use these characterisations as definitions.

\begin{thm}\label{thm:nonauto}
None of the sequences $t$, $\rho$ and $\tau$ of Theorems~\ref{thm:main},
\ref{thm:three-initial} and~\ref{thm:four} is $k$-automatic for any integer $k\ge2$.
\end{thm}

For each of the three sequences, we choose a letter on the line $y=x+1$ whose
gaps are unbounded but have length $O(\log(x+2))$, where $x$ is the starting position
of a gap. Lemma~\ref{lem:gaps} shows that an automatic sequence with unbounded gaps
must contain arbitrarily late intervals without the chosen letter whose lengths
are proportional to their starting positions.
These two properties are incompatible.

\begin{lem}\label{lem:line}
If $T:\N^2\to A$ is $k$-automatic, then so is the sequence $u(x)=T(x,x+1)$.
\end{lem}
\begin{proof}
Let $e\ge0$ and $0\le b<k^e$. If $b+1<k^e$, then $u(k^ex+b)=T(k^ex+b,k^ex+b+1)$ is the diagonal
$x\mapsto T'(x,x)$ of the element $T'(x,y)=T(k^ex+b,k^ey+b+1)$ of the $k$-kernel of $T$. If
$b=k^e-1$, then $u(k^ex+b)=T(k^ex+k^e-1,k^e(x+1))$ is $x\mapsto T''(x,x+1)$ for the element
$T''(x,y)=T(k^ex+k^e-1,k^ey)$ of the $k$-kernel of $T$. Hence the $k$-kernel of $u$ has at most
twice as many elements as that of $T$.
\end{proof}

\begin{lem}\label{lem:gaps}
Let $u:\N\to A$ be $k$-automatic, let $N$ be the number of elements of its $k$-kernel, and let the
letter $a$ occur in $u$. If $u(n)\ne a$ for all $n$ in an interval of length $2k^m$ with $m\ge2^N$,
then there are integers $q\ge1$ and $p\ge1$ such that $u(n)\ne a$ for all
$n\in[qk^{m+tp},(q+1)k^{m+tp})$ and all $t\ge0$.
\end{lem}
\begin{proof}
An interval of length $2k^m$ contains an interval $[qk^m,(q+1)k^m)$ with $q\ge0$. For a sequence
$v$ in the $k$-kernel $\mathcal{K}$ of $u$ and a digit $b<k$, the sequence $n\mapsto v(kn+b)$ lies in $\mathcal{K}$.
Let $\mathcal{K}_i\subseteq \mathcal{K}$ be the set of the sequences $n\mapsto u(k^in+b)$ with $0\le b<k^i$. Writing
$b=k^ib_1+b'$ with $b_1<k$ and $b'<k^i$ shows that $\mathcal{K}_{i+1}$ consists of the sequences
$n\mapsto v(kn+b_1)$ with $v\in \mathcal{K}_i$ and $b_1<k$, so $\mathcal{K}_{i+1}$ is determined by $\mathcal{K}_i$. Among the
$2^N+1$ sets $\mathcal{K}_0,\dots,\mathcal{K}_{2^N}$ two coincide, say $\mathcal{K}_i=\mathcal{K}_{i+p}$ with $i+p\le2^N\le m$, and then
$\mathcal{K}_{m+tp}=\mathcal{K}_m$ for all $t\ge0$. By assumption $v(q)\ne a$ for all $v\in \mathcal{K}_m$, hence for all
$v\in \mathcal{K}_{m+tp}$, which is the claim. If $q=0$, the claim says that $u(n)\ne a$ for all
$n<k^{m+tp}$ and all $t$, so $a$ would not occur in $u$. Hence $q\ge1$.
\end{proof}

\begin{proof}[Proof of Theorem~\ref{thm:nonauto}]
Suppose that $t$ is $k$-automatic. By Lemma~\ref{lem:line} the sequence
$u(x)=t(x,x+1)=T\cl{2x+1}1$ is $k$-automatic. By Theorem~\ref{thm:main}, $T\cl s1=0$ exactly when
$s=3$, by rule~\ref{D0}, or when $\cl s1$ lies on the border of a box of $\cB$, by rule~\ref{D1},
since the other rules give the values $1$ and $2$ on odd rows. The cells $\cl s1$ on borders are
the cells of the edges, because the walls $|d|=1$ belong to the boxes of level $1$, whose rows are
their two edges and the row between them. Inside the main box $S_n$ two consecutive edge times
differ by at most $4n+2$, which is the length of the switching gap of $S_n$, and consecutive main
boxes are $4$ rows apart. The initial runs at $x=0$ and $x=2,3$ are finite exceptions.
Every later maximal run of terms of $u$ different from $0$, starting at $x_0\ge5$,
lies between two consecutive edge times $s'=2x_0-1$ and $s''$, and its length is
$(s''-s')/2-1\le2n$, where $n$ is the level of the main box that contains $s'$ or ends just before
it. Since $e_n=2^{n+2}-2n+3>2^{n+1}$, we have $n<\log_2s'$, and the run has length less than
$2\log_2(2x_0)$. On the other hand the switching gap of $S_n$ gives a run of length $2n$, so the
runs are unbounded. Lemma~\ref{lem:gaps}, applied with $a=0$, gives runs of length $k^{m+tp}$
starting at $qk^{m+tp}$ for all $t\ge0$, which contradicts the bound $2\log_2(2qk^{m+tp})$ for
large $t$. Hence $t$ is not $k$-automatic.

For $\tau$ we use the letter $3$ on the same line. By Theorem~\ref{thm:four}(a),
the cells $\tau\cl s{\pm1}$ with $s\ge56$ are green at offset $12$ in each box
of level $1$ and $12$ rows after each middle block begins. These are the only green
entries on $|d|=1$ in the tables $P_1$ and $\mathcal{M}$. None of the other description rules
gives a green cell on this line: borders, transition rows and junction rows give
$0$ or $1$, rules (iv) and (F6) concern even columns, $I$ takes the values $2,0,0$
on the wall column $a=1$, and $\Lam$ is white on $|d|=1$.

The middle block of a box of level $n\ge2$ starts at offset $c_{\mathrm{mid}}=\ell_n/2-7$.
Thus every box of level $n\ge1$ has a green cell on $d=1$ at offset
$\ell_n/2+5$, five rows after the middle of the box. The same statement holds
for each virtual box $V_j$ at its level.
Induction on $n$, using the three sub-boxes, shows that each row of a box of level
$n\ge1$ lies within $4n+12$ rows of a green cell on $d=1$ inside the box.
Consecutive main boxes are at most $12$ rows apart. The main boxes that meet the
rows up to $s$ have levels at most $\log_3s+1$: in block $j$, the row index satisfies
$s\ge E_j-1\ge30\cdot27^j$ and the levels are at most $3j+4$.
Consequently a maximal run of terms of $u(x)=\tau(x,x+1)$ different from $3$,
starting at a sufficiently large $x_0$, has length at most $C\log x_0$ for a constant $C$.

These runs are unbounded. In a box of level $n\ge3$, consider the last row with
a green cell on $d=1$ in its first child and the first such row in its second child.
The two rows are $8n-2$ rows apart, with no green cell on $d=1$ between them.
Lemma~\ref{lem:gaps}, applied with $a=3$, contradicts the logarithmic bound as in the
case of $t$.

Finally consider $v(x)=\rho(x,x+1)=1-U\cl{2x+1}1$, and use the red letter $1$.
By Theorem~\ref{thm:three-initial}, its occurrences are exactly the intersections with
the borders of odd-level boxes: the other zero cells of $U$ have $d=0$, even-level
edges have even time, and the wall $d=1$ of a level-$1$ box consists only of its two
edge cells.

For a box of level $n$, let $Z_n$ be the relative times of these red cells on $d=1$.
Then $Z_1=\{0,2\}$, $Z_2=\varnothing$ and $Z_3=\{0,14\}$. For $n\ge3$, the first
red time and the distance of the last red time from $b_n$ are both
\[
a_n=\begin{cases}0,&n\text{ odd},\\3,&n\text{ even},\end{cases}
\]
and the largest difference between consecutive red times is $2n+8$.
To prove these assertions, use the child list of Section~\ref{sec:three-initial}.
For odd $n$, the red cells occur on both edges of the parent. For even $n\ge4$,
the first and last children have odd level and lie $3$ rows from the parent's edges.
Within either half, consecutive
children have successive levels. Their nearest red cells are $6$ rows apart, except
when the empty set $Z_2$ is skipped: the level-$2$ duration $b_2=6$ and its two
neighbouring gaps of $3$ give $12$ rows between the red cells of levels $3$ and $1$.
For an odd-level parent, the gap from its opening to its first red child cell is $12$
at level $5$ and $6$ at odd levels at least $7$, and closing is symmetric.
The two central level-$1$ children, for $n\ge4$, have red edges separated by exactly
$2n+8$, with no red cell between them. This gap exceeds the other gaps and, by induction,
all gaps inside a child. The case $n=3$ follows from $Z_3$ directly.

Consequently $C_n$, for $n\ge3$, contains a run of exactly $n+3$ non-red terms of $v$,
and no longer run between its red cells. Between $C_n$ and $C_{n+1}$, for $n\ge3$,
the nearest red times are $a_n+3+a_{n+1}=6$ rows apart. The initial level-$2$ main box
has no red cell on this line. The surrounding red times $7$ and $19$ give a single
finite exception of $12$ rows.

Since the durations increase, $b_n\ge3b_{n-3}$ for $n\ge4$, and hence
\[
b_m\ge2\cdot3^{\lfloor(m-1)/3\rfloor},\qquad
c_n\ge b_{n-1}\ge2\cdot3^{(n-4)/3}\quad(n\ge2).
\]
A maximal non-red run starting at $x_0\ge10$ has its preceding red cell at time
$s'=2x_0-1$ in a main box $C_n$, with $n\ge3$ and $c_n\le s'$. Its length is at most
\[
n+3\le3\log_3(s'/2)+7<3\log_3x_0+7.
\]
Thus the non-red runs are unbounded, but their lengths are $O(\log(x_0+2))$.
If $\rho$ were $k$-automatic, Lemma~\ref{lem:line} would make $v$ $k$-automatic.
Lemma~\ref{lem:gaps}, with $a=1$, would then give non-red intervals of length
$k^{m+tp}$ starting at $qk^{m+tp}$ for all $t\ge0$. Extend each to a maximal run:
its starting point is no larger, so its length still has a logarithmic upper bound
in $qk^{m+tp}$. This contradicts the displayed interval lengths for large $t$.
\end{proof}

\section{Computer verification}\label{sec:computer}

The computations below test the recurrence, geometry and counting formulas
against the explicit descriptions. Verification code, finite tables and
computation logs accompany the paper.\footnote{\href{https://github.com/bogdan27182/paper-vn}{GitHub repository}.}

The proofs in Sections~\ref{sec:proof}--\ref{sec:triples} are independent of these
computations. For three initial symbols, Lemma~\ref{lem:three-initial-local}
reduces the recurrence check to a finite catalogue, whose verification is part
of the proof. Claim 12 gives additional checks on larger regions. For four
letters, Theorem~\ref{thm:four}(a) and Corollary~\ref{cor:four-triples} use claim 6,
while Theorem~\ref{thm:four}(b) uses claims 8 and 14 near the origin.
The exhaustive search for Theorem~\ref{thm:three-none} is specified in its proof.
The non-automaticity proofs in Section~\ref{sec:nonauto} follow from the
descriptions without further computation.
\begin{enumerate}
\item The description $D$ of Definition~\ref{def:D}, implemented without the recurrence, agrees
with $t$, computed from the recurrence, at all $2^{34}$ cells $(x,y)$ with $0\le x,y<2^{17}$.
This range contains the main boxes $S_1,\dots,S_{15}$ completely.
\item $D$ is correct, in the sense of Section~\ref{sec:rot}, at all $8.5\cdot10^9$ cells of the
band $|d|\le w_{25}+8$ with $x+y<1.35\cdot10^8$, which contains $S_1,\dots,S_{24}$ completely.
\item The set of triples of Theorem~\ref{thm:triples}(a) and the occurrence statements of
(c) were confirmed on the same $2^{17}\times2^{17}$ square.
\item The coordinate formulas of Theorem~\ref{thm:structure}, the structure of the transition
squares $T_1,\dots,T_9$ including the runs of the staircases, and the structure of the $162$
transition squares and $170$ switching squares inside the ordinals up to $S_9$ were checked cell
by cell.
\item The classification of Lemma~\ref{lem:Eclasses} was checked on all cells with $x+y<2400$:
every environment cell with an irregular parent or on the diagonal, $12\,709$ cells in total,
belongs to one of the classes (E1) to (E6).
\item $\Delta$ is correct at all $55\,492\,541$ cells with $|d|\le s-2$ of the region $\mathcal{R}$ of
Lemma~\ref{lem:windows}, and the argument triples of $g$ at these cells were recorded.
The computation used checked integer initialization and undefined-behaviour sanitization
(Proposition~\ref{prop:four-machine}(a),(b)).
\item $\Delta$ agrees with $\tau$, computed from the recurrence, at all cells with $x,y<40000$
(Proposition~\ref{prop:four-machine}(c)).
\item The white-bordered rectangles of $\Delta$ centred on the diagonal, with a non-white cell
strictly inside each row between their edges, with $s<2\cdot10^5$ and half-width at most $100$,
are exactly the $33\,327$ boxes of $\mathcal C$ in this range. This range contains all ordinals
whose top row lies in $s<58$.
\item Proposition~\ref{prop:triple} was checked cell by cell for $2\le n\le11$, and
Proposition~\ref{prop:diagonal}(a),(b) for $0\le j\le3$. The facts (i) and (ii) of
Section~\ref{sec:four-bc} were checked at all cells with $x,y\ge16$, $56\le s<2\cdot10^6$ and
$|d|\le200$.
\item On the line $y=x+1$, for $x<4095$, the white cells of $t$ are the intersections
with edges of boxes of $\cB$ and the cell $(1,2)$.
On the same line, the green cells of $\tau$ with $s\ge56$ occur on the rows used in
the proof of Theorem~\ref{thm:nonauto}. This was checked for $x<11999$ using the
recurrence and for $s<5\cdot10^7$ using $\Delta$.
\item For $\tau$, the ordinals contained in no other ordinal were enumerated on the
anti-diagonals $s<3\cdot10^7$ in the band $|d|\le100$. Cells outside the band were
taken from $\Lam$. The levels and positions of the enumerated ordinals agree with
Theorem~\ref{thm:four}(c) up to the box of level $13$.
\item For the three-initial-symbol sequence, $U$ agrees with $u=1-\rho$, computed from
the recurrence, on all cells with $x,y<2048$. The local recurrence for $U$ was also checked
for $2\le s<60000$ and $|d|\le44$, using symmetry. In a $1536\times1536$ recurrence grid,
all $276$ zero-bordered rectangles with half-width at most $24$ are exactly the boxes
of $\mathcal H$ that fit in the grid. Their counts inside $C_1,\dots,C_{11}$ agree with
Corollary~\ref{cor:three-count}.
\item The ordinary--virtual--ordinary assembly of Proposition~\ref{prop:triple}, including
the level-$0$ exception, was checked in both $P_n$ and $Q_n$ for $1\le n\le8$ on
$1\,622\,784$ constituent cells. The identity
$\Delta\cl sd=Q_{3j+3}(s-E_j+4,d)$ was checked on all $4\,542\,171$ cells of
$V_0,V_1,V_2$.
\item \label{check:arithmetic-and-boundaries}The arithmetic formula of Corollary~\ref{cor:three-count} agrees with its recurrence
for $1\le n\le200$. For levels $1$ through $18$, the positions of the red cells on
$d=1$ agree with the odd-level edge description in the non-automaticity proof, including
the largest time gap $2n+8$ for $n\ge3$. The $233$ red cells of $\rho(x,x+1)$ in a
$1024\times1024$ recurrence grid also agree with this description. For the four-letter
sequence, all $53$ centred white opening segments with $1\le s<58$ have their first
non-white wall cell by time $111$, which justifies the finite-origin cutoff in part (b).
The two separator cells of Lemma~\ref{lem:four-margins} were checked for
$58\le s<613\,224$, and the $204\,387$ edge times in that range all have non-white
cells immediately beyond their widest edge.
\end{enumerate}

The finite catalogue in Lemma~\ref{lem:three-initial-local} contains $7930$
horizontal-edge windows, $1570$ wall windows and $858$ central windows,
including repeated value quadruples from different contexts. Every window
satisfies the local recurrence.

The maps in Lemma~\ref{lem:windows} were checked against the cell definitions
on all $123\,540$ frame windows at levels $6$ through $8$, on representatives
at levels $9$ through $30$, and on $8190$ representatives near the block
boundaries in blocks $3$ through $9$. These finite checks supplement the proof
that the maps preserve values at arbitrary levels.

For the search in Theorem~\ref{thm:three-none}, the rectangle detector was
checked against an independent direct enumeration. As a positive control, it
accepts the level-$2$ box of the four-letter sequence, together with its eight
descendants.

\appendix
\section{The finite tables of the description \texorpdfstring{$\Delta$}{Delta}}\label{app:tables}

The tables used in Section~\ref{sec:four-desc} follow. Their entries were read from
$\tau$, computed by the recurrence. In $P_0$ and $P_1$, the row index $r$ is the
offset from the opening edge. In $\mathcal{M}$, rows are indexed from the closing offset $c_{\mathrm{mid}}$
of the preceding child of level $1$. The columns give positions $d$ from $-v$ to $v$.
A dot marks a position of the wrong parity, which carries no cell.

The pattern $P_0$, rows $r=0,1,2$, columns $d=-1,0,1$:
\begin{verbatim}
0.0
.1.
0.0
\end{verbatim}
The pattern $P_1$, rows $r=0,\dots,14$, columns $d=-5,\dots,5$:
\begin{verbatim}
 0  0.0.0.0.0.0
 1  .1.1.1.1.1.
 2  0.1.1.1.1.0
 3  .2.0.0.0.2.
 4  0.2.0.0.2.0
 5  .3.2.1.2.3.
 6  0.3.0.0.3.0
 7  .2.1.0.1.2.
 8  0.1.0.0.1.0
 9  .2.2.1.2.2.
10  0.2.0.0.2.0
11  .3.3.0.3.3.
12  0.3.3.3.3.0
13  .2.0.2.0.2.
14  0.0.0.0.0.0
\end{verbatim}
The middle block $\mathcal{M}$, rows $r-c_{\mathrm{mid}}=0,\dots,14$, columns $d=-5,\dots,5$:
\begin{verbatim}
 0  0.0.0.0.0.0
 1  .1.1.1.1.1.
 2  0.1.1.1.1.0
 3  .2.0.0.0.2.
 4  0.2.0.0.2.0
 5  .3.2.1.2.3.
 6  2.3.0.0.3.2
 7  .2.1.0.1.2.
 8  0.1.0.0.1.0
 9  .2.2.1.2.2.
10  0.2.0.0.2.0
11  .3.3.0.3.3.
12  2.3.3.3.3.2
13  .2.0.2.0.2.
14  0.0.0.0.0.0
\end{verbatim}
The strip $x\le13$: each row gives $x$, the threshold $y_0(x)$ and the values
$\tau(x,24),\dots,\tau(x,35)$. For $y\ge y_0(x)$ the values repeat with period $12$, that is,
$\tau(x,y)=\tau\big(x,\,24+((y-24)\bmod12)\big)$. The strip $y\le13$ is symmetric.
\begin{verbatim}
 x  y0
 0   0  111111111111
 1   0  102102102102
 2   1  220220220220
 3   3  031031031031
 4   3  203321203321
 5   4  232213011002
 6   2  332020110010
 7   4  021201100112
 8   4  302231001102
 9   4  320330011002
10   4  112020110010
11  12  022301100112
12  13  303320222022
13  18  320212033103
\end{verbatim}
The corner: $\tau(x,y)$ for $0\le x,y<20$, with $x$ the row and $y$ the column.
\begin{verbatim}
 0  11111111111111111111
 1  10210210210210210210
 2  12022022022022022022
 3  11203103103103103103
 4  10232120332120332120
 5  12011301100223221301
 6  11202011001033202011
 7  10230110011202120110
 8  12013100110230223100
 9  11203001100232033001
10  10232011001011202011
11  12011202220002230110
12  11202230331010332022
13  10230332021200300103
14  12013221202233201120
15  11203202230330011023
16  10232120332020110023
17  12011301100101100101
18  11202011001120222000
19  10230110011023033101
\end{verbatim}
The initial segment: $\tau\cl sd$ for $28\le s\le55$, each row listing the positions $d=0$ or
$d=1$, according to the parity of $s$, up to $d=s-28$. The values at negative $d$ follow by
symmetry, and the cells with $s<56$ that are not listed lie in the corner or in the strips.
\begin{verbatim}
28  2
29  0
30  11
31  11
32  002
33  020
34  1232
35  0301
36  01222
37  01020
38  122332
39  020301
40  0332222
41  3301020
42  20222332
43  00020302
44  111332233
45  111301032
46  0000222222
47  0000020101
48  11111332222
49  11111302020
50  000000233332
51  000000032301
52  1111111222222
53  1111110101020
54  00000022222332
55  00000202020302
\end{verbatim} 

\section*{Acknowledgments}

This paper was produced in an AI-assisted workflow. AI assistants helped
explore candidate constructions, draft proofs, and write software for numerical experiments.

\end{document}